\documentclass[12pt]{article}
\usepackage{amsmath,mathrsfs,amsfonts,amsthm}
\usepackage{mathtools}
\usepackage{stmaryrd}
\usepackage{amssymb}
\usepackage{fullpage}
\usepackage{hyperref}
\usepackage{epic}
\usepackage{eepic}
\usepackage[english]{babel}
\usepackage[all]{xy}
\usepackage{enumerate}
\usepackage[latin1]{inputenc}

\newenvironment{paragr}[1][]{\refstepcounter{subsubsection} \noindent \textbf{\thesubsubsection . \ #1}}{\medskip}

\newcounter{keepeqno}

\usepackage{mathabx}
\usepackage{dsfont}

\newtheorem{theo}{Theorem}[section]
\newtheorem{lem}{Lemma}[section]
\newtheorem{prop}{Proposition}[section]

\newtheorem{rem}{Remark}[section]

\newtheorem{conj}{Conjecture}[section]

\numberwithin{equation}{section}

\def\tM{\widetilde{M}}
\def\tm{\widetilde{m}}
\def\bC{\mathbb{C}}
\def\ou{\overline{u}}
\def\oN{\overline{N}}
\def\Alt{\mathrm{Alt}}
\def\tf{\widetilde{f}}
\def\cN{\mathcal{N}}
\def\tpi{\widetilde{\pi}}
\def\Unit{\mathrm{Unit}}

\DeclareMathOperator{\orth}{orth}

\DeclareMathOperator{\Tr}{Trace}
\DeclareMathOperator{\Tra}{Tr}

\DeclareMathOperator{\Temp}{Temp}

\DeclareMathOperator{\Ker}{Ker}
\DeclareMathOperator{\ad}{ad}
\DeclareMathOperator{\Ad}{Ad}

\DeclareMathOperator{\Lie}{Lie}

\DeclareMathOperator{\vol}{vol}
\DeclareMathOperator{\Supp}{Supp}
\DeclareMathOperator{\Gal}{Gal}

\DeclareMathOperator{\St}{St}

\DeclareMathOperator{\Irr}{Irr}

\DeclareMathOperator{\Stab}{Stab}

\DeclareMathOperator{\cA}{\mathcal{A}}
\DeclareMathOperator{\bR}{\mathbb{R}}
\DeclareMathOperator{\bZ}{\mathbb{Z}}

\DeclarePairedDelimiter\ceil{\lceil}{\rceil}
\DeclarePairedDelimiter\floor{\lfloor}{\rfloor}
\DeclareMathOperator{\cC}{\mathcal{C}}

\DeclareMathOperator{\ind}{ind}

\DeclareMathOperator{\stab}{stab}

\DeclareMathOperator{\cU}{\mathcal{U}}

\DeclareMathOperator{\cO}{\mathcal{O}}

\DeclareMathOperator{\Bil}{Bil}

\DeclareMathOperator{\Sym}{Sym}

\DeclareMathOperator{\Id}{Id}
\DeclareMathOperator{\Norm}{Norm}

\DeclareMathOperator{\rs}{rs}

\DeclareMathOperator{\disc}{disc}
\DeclareMathOperator{\cM}{\mathcal{M}}

\DeclareMathOperator{\GL}{\mathrm{GL}}

\DeclareMathOperator{\sr}{sr}
\DeclareMathOperator{\Sh}{Sh}

\DeclareMathOperator{\SO}{\mathrm{SO}}
\DeclareMathOperator{\rO}{\mathrm{O}}
\DeclareMathOperator{\Sp}{\mathrm{Sp}}
\DeclareMathOperator{\SL}{\mathrm{SL}}
\DeclareMathOperator{\bdd}{\mathrm{bdd}}
\DeclareMathOperator{\SI}{\mathrm{SI}}
\DeclareMathOperator{\rI}{\mathrm{I}}
\DeclareMathOperator{\rJ}{\mathrm{J}}
\DeclareMathOperator{\rk}{\mathrm{rk}}

\newcommand{\eq}[1][r]
{\ar@<-3pt>@{-}[#1]
	\ar@<-1pt>@{}[#1]|<{}="gauche"
	\ar@<+0pt>@{}[#1]|-{}="milieu"
	\ar@<+1pt>@{}[#1]|>{}="droite"
	\ar@/^2pt/@{-}"gauche";"milieu"
	\ar@/_2pt/@{-}"milieu";"droite"}

\makeatletter
\newlength{\doublefracgap}
\DeclareRobustCommand{\doublefrac}[2]{%
	\mathinner{\mathpalette\doublefrac@{{#1}{#2}}}%
}
\newcommand{\doublefrac@}[2]{\doublefrac@@#1#2}
\newcommand{\doublefrac@@}[3]{%
	\ooalign{%
		\raisebox{\doublefracgap}{$\m@th#1\frac{#2}{\phantom{#3}}$}\cr
		\raisebox{-\doublefracgap}{$\m@th#1\frac{\phantom{#2}}{#3}$}\cr
	}%
}

\makeatother

\title{The Hiraga-Ichino-Ikeda conjecture on formal degrees for classical groups}
\author{Rapha\"el Beuzart-Plessis}

\begin{document}
	
	\maketitle
	
	\begin{abstract}
		We prove a conjecture of Hiraga-Ichino-Ikeda relating formal degrees of square-integrable representations to adjoint gamma factors for symplectic and even orthogonal groups over characteristic zero non-Archimedean local fields. The proof is based on the twisted endoscopic characterization of the local Langlands correspondence for such groups and extends an approach already appearing in the original paper of Hiraga-Ichino-Ikeda \cite{HII} itself inspired from earlier work of Shahidi \cite{Shaend}. Namely, the argument consists in computing in two different ways a residue of certain standard intertwining operators suitably extended to a large representation realized on an explicit space of functions. This gives rise to the spectral decomposition of certain singular twisted orbital integrals. The main theorem follows from a comparison of this identity with the Plancherel formula of a symplectic or even orthogonal group. The proof can be adapted to deal with other classical groups, namely odd orthogonal and unitary groups, but the formal degree conjecture was already established in these cases. 
	\end{abstract}
	
\section{Introduction}

\subsection{The formal degree conjecture}

Let $\underline{G}$ be a connected reductive group defined over a local field $F$ and $G=\underline{G}(F)$. Recall that an admissible irreducible representation $\pi$ of $G$ is said to be square-integrable if its central character is unitary and its matrix coefficients are square-integrable over $G/A_G$, where $A_G$ denotes the maximal split torus in the center of $G$. To a square-integrable representation $\pi$, we can attach its formal degree $d(\pi)$ which is the unique positive real number such that
\begin{equation*}
\displaystyle \int_{G/A_G} \langle \pi(g)v,v^\vee\rangle \langle w,\pi^\vee(g)w^\vee\rangle \; \frac{dg}{da}=\frac{\langle v,w^\vee\rangle \langle w,v^\vee\rangle}{d(\pi)}, \; \mbox{ for every } v,w\in V_\pi,\; v^\vee,w^\vee\in V_\pi^\vee,
\end{equation*}
where $V_\pi$ denotes the space of (smooth vectors of) $\pi$ and $V_\pi^\vee$ the space of the contragredient representation. This invariant however depends a priori on the choice of a Haar measure on $G/A_G$. Following \cite{HII}, we can choose such a Haar measure by fixing a non trivial additive character $\psi: F\to S^1$. More precisely, denoting by $\underline{G}_{\bZ}$ the split form of $\underline{G}$ over $\bZ$, up to $\pm 1$ there is a unique invariant volume form $\omega_{\bZ}$ on $\underline{G}_{\bZ}$ that is nowhere vanishing. Pulling back $\omega_{\bZ}$ by any isomorphism $\underline{G}_{F^{sep}}=\underline{G}\times_F F^{sep}\simeq \underline{G}_{\bZ}\times F^{sep}$ over a separable closure $F^{sep}$ of $F$ gives rise to a nonzero invariant volume form $\omega_{F^{sep}}$ on $\underline{G}_{F^{sep}}$, canonical up to $\pm 1$. We can then pass from $\omega_{F^{sep}}$ to a Haar measure $\lvert\omega_{F^{sep}}\rvert_\psi$ on $G$ by Weil's construction \cite{Weiladeles}: if in local coordinates $x_1,\ldots,x_d$ we have
$$\displaystyle \omega_{F^{sep}}=f(x_1,\ldots,x_d)dx_1\wedge\ldots\wedge dx_d,$$
then
$$\displaystyle \lvert\omega_{F^{sep}}\rvert_\psi=\lvert f(x_1,\ldots,x_d)\rvert d_\psi x_1\ldots d_\psi x_d$$
where $\lvert .\rvert$ denotes the normalised absolute value on $F$, extended to $F^{sep}$, and $d_\psi x$ stands for the $\psi$-self dual Haar measure on $F$. Of course, the $\pm 1$ ambiguity in the definition of $\omega_{F^{sep}}$ disappears when we pass to $\lvert \omega_{F^{sep}}\rvert_\psi$ and we shall denote the resulting Haar measure, which depends only on $\psi$, by $d_\psi g$.

We construct similarly a Haar measure $d_\psi a$ on $A_G$. We can then take $dg=d_\psi g$, $da=d_\psi a$ in the definition of the formal degree, that we will denote by $d_\psi(\pi)$ to emphasize the dependence on $\psi$.

\subsection{The conjecture of Hirage-Ichino-Ikeda}

In \cite{HII}, Hiraga, Ichino and Ikeda have proposed a conjectural formula for $d_\psi(\pi)$ in terms of the local Langlands parameterization for $G$. To be more specific, let us assume that $\underline{G}$ admits a pure inner form $\underline{G}_{qd}$ that is quasi-split. We recall that the pure inner forms of $\underline{G}$ are naturally parameterized by the Galois cohomology set $H^1(\Gamma_F, \underline{G}(F^{sep}))$ (where $F^{sep}$ denotes a separable closure of $F$ and $\Gamma_F=\Gal(F^{sep}/F)$ the absolute Galois group) and that the assumption is satisfied by all classical groups (that is, symplectic, special orthogonal and unitary groups). 

The local Langlands correspondence predicts the existence of a natural decomposition
$$\displaystyle \Irr(G)=\bigsqcup_\phi \Pi^G(\phi)$$
of the admissible dual $\Irr(G)$ of $G$ into $L$-packets $\Pi^G(\phi)$. The latter are finite sets indexed by equivalence classes of $L$-parameters that is group morphisms
$$\displaystyle \phi: W'_F\to {}^L G,$$
where $W'_F$ denotes the Weil-Deligne group of $F$ (just its Weil group $W_F$ if $F$ is Archimedean and $W_F\times \SL_2(\bC)$ if $F$ is non-Archimedean) and ${}^L G=\widehat{G}\rtimes \Gamma_F$ the $L$-group of $G$, that are continuous, semisimple, algebraic on $\SL_2(\bC)$ and commuting with the natural projections to $\Gamma_F$; two $L$-parameters $\phi$, $\phi'$ being equivalent if they are conjugate by an element of the dual group $\widehat{G}$. 

We also expect that each $L$-packet $\Pi^G(\phi)$ can be parametrized by a certain subset of irreducible representations of the (finite) group of components
$$\displaystyle S_\phi:=\pi_0(Z_{\widehat{G}}(\phi))$$
of the centralizer $Z_{\widehat{G}}(\phi)$ of $\phi$ in $\widehat{G}$. That is to say there should exists an injective map
$$\displaystyle \Pi^G(\phi)\hookrightarrow \Irr(S_\phi),$$
$$\displaystyle \pi\mapsto \rho_{\pi,\varpi},$$
which, as the notation suggests, depends on an auxiliary choice $\varpi$ which is that of a Whittaker datum for one of the quasi-split pure inner form $G_{qd}$ of $G$\footnote{Recall that a Whittaker datum is a pair $(N,\xi)$ consisting of the unipotent radical $N$ of a Borel subgroup (defined over $F$) and of a non-degenerate character $\xi$ of $N$, taken up to $G_{qd}$-conjugacy.}.

Let $A_{\widehat{G}}=Z(\widehat{G})^{\Gamma_F,0}$ be the neutral component of the subgroup of $\Gamma_F$-fixed points in the center of $\widehat{G}$. We write $\Ad_G$ for  the adjoint representation of ${}^L G$ on $\Lie(\widehat{G})/\Lie(A_{\widehat{G}})$. For $\pi\in \Irr(G)$, let us denote by $\phi_\pi$ its $L$-parameter (i.\ e.\ $\pi\in \Pi^G(\phi_\pi)$) and let
$$\displaystyle \gamma(s,\pi,\Ad_G,\psi):= \gamma(s,\Ad_G\circ \phi_\pi,\psi)=\epsilon(s,\Ad_G\circ \phi_\pi,\psi)\frac{L(1-s,\Ad\circ \phi_\pi)}{L(s,\Ad\circ \phi_\pi)}$$
be its adjoint gamma factor. Here, $L(s,\Ad_G\circ \phi_\pi)$ denotes the local $L$-factor of the Weil-Deligne representation $\Ad_G\circ \phi_\pi$ whereas $\epsilon(s,\Ad_G\circ \phi_\pi,\psi)$ stands for the local $\epsilon$-factor.

The main conjecture of \cite{HII} on formal degrees can now be stated as follows.

\begin{conj}[Hiraga-Ichino-Ikeda]\label{Conj HII}
Let $\pi$ be a square-integrable irreducible representation of $G$. Then, we have
\begin{equation*}
\displaystyle d_\psi(\pi)=\frac{\deg(\rho_{\pi,\varpi})}{[X^*(A_G):X^*(G)]\lvert S_{\phi_\pi}\rvert} \lvert \gamma(0,\pi,\Ad_G,\psi)\rvert,
\end{equation*}
where $\deg(\rho_{\pi,\varpi})$ denotes the degree of the representation $\rho_{\pi,\varpi}$ of $S_{\phi_\pi}$.
\end{conj}

\begin{rem}
The conjecture includes implicitely the fact that $\gamma(s,\pi,\Ad_G,\psi)$ is regular at $s=0$. However, it is expected that, for $\pi$ square-integrable, the $L$-parameter $\phi_\pi$ is discrete, which means that its centralizer $Z_{\widehat{G}}(\phi_\pi)$ is finite modulo $A_{\widehat{G}}$, and this last condition is readily seen to imply that $\gamma(s,\pi,\Ad_G,\psi)$ is regular for $\Re(s)\geq 0$.

Since the formal degree $d_\psi(\pi)$ is obviously independent of the Whittaker datum $\varpi$, the conjecture also implies that the degree $\deg(\rho_{\pi,\varpi})$ doesn't depend on $\varpi$ either. There is however a precise expectation on the dependence of $\rho_{\pi,\varpi}$ on $\varpi$, see \cite[\S 9, (3)]{GGP}, namely changing the Whittaker datum should result in twisting $\rho_{\pi,\varpi}$ by a certain one-dimensional character $S_{\phi_\pi}\to \bC^\times$; an operation that clearly doesn't affect the degree of the representation.

Finally, the denominator in the above formula is slightly different from the one in \cite{HII}. Namely, in {\em op.\ cit.\ } the factor in the denominator is the cardinality of the centralizer $S_{\phi_\pi}^\natural$ in the dual group of $G^\natural:=G/A_G$. However it is easy to see that for discrete $L$-parameters we have
$$\displaystyle \lvert S_{\phi_\pi}^\natural\rvert=[X^*(A_G):X^*(G)]\lvert S_{\phi_\pi}\rvert.$$ 
We also emphasize that this conjecture translates into a more natural one, that is without the factor $[X^*(A_G):X^*(G)]$, for the Plancherel density (suitably normalised) of square-integrable representations; see the next section.
\end{rem}

\subsection{Plancherel density}

When combined with a conjecture of Langlands \cite[Appendix II]{LangEis} on the normalization of standard intertwining operators, Conjecture \ref{Conj HII} leads to a conjectural formula for the Plancherel density of $G$. More precisely, we recall that the Plancherel formula for $G$ takes the form \cite{WaldPlanch}
$$\displaystyle f(1)=\int_{\Temp(G)} \Theta_\pi(f) d\mu_G(\pi),\;\;\; f\in C_c^\infty(G),$$
where $\Temp(G)$ denotes the tempered dual of $G$, that is the set of isomorphism classes of irreducible tempered representations of $G$, $\Theta_\pi(f):=\Tra \pi(f)$ is the distributional character of $\pi$ and $d\mu_G(\pi)$ is a certain measure on $\Temp(G)$, the so-called Plancherel measure. We emphasize that the character $\Theta_\pi$ depends on a chosen Haar measure on $G$ (through the formation of the operator $\pi(f)=\int_G f(g)\pi(g)dg$) and therefore, under our current normalization, on the additive character $\psi$.

The Plancherel density $\mu_G(\pi)$ appears when we compare the Plancherel measure $d\mu_G(\pi)$ to some natural, more elementary, measure $d\pi$ on $\Temp(G)$. Namely, recall that the irreducible tempered representations of $G$ are exactly the irreducible constituents of normalised parabolic inductions $\rI_P^G(\sigma)$, where $P=MU$ is a parabolic subgroup of $G$ and $\sigma$ a square-integrable representation of $M$. For such $P$ and $\sigma$, the group $X^u(M)$ of unitary unramified characters of $M$ acts by twisting on square-integrable representations:
$$\displaystyle (\lambda,\sigma)\mapsto \sigma_\lambda:=\sigma\otimes \lambda.$$
This group is a compact real torus and roughly speaking $d\pi$ is the unique measure such that the map
\begin{equation*}
	\displaystyle X^u(M)\to \Temp(G),\; \lambda\mapsto [\rI_P^G(\sigma_\lambda)]
\end{equation*}
is locally measure preserving for every $P=MU$ and $\sigma$ as before, where $[\rI_P^G(\sigma_\lambda)]$ stands for the isomorphism class of $\rI_P^G(\sigma_\lambda)$ and we endow $X^u(M)$ with a certain Haar measure. However, one subtlety here is that the representation $\rI_P^G(\sigma_\lambda)$ is not always irreducible and the above map is not always locally injective. These two phenomena are actually related and we refer the reader to \S \ref{S spectral measure} for a precise definition. We simply emphasize here that the Haar measure with which we need to endow $X^u(M)$ is not necessarily the obvious one, that is the one of total mass $1$, but should rather be defined as follows. Let $\underline{A}^M$ be the maximal split torus that is a quotient of $\underline{M}$ and $A^M$ be its group of $F$-points (thus, with notation from the previous section, the dual torus $\widehat{A^M}$ identifies naturally with the central torus $A_{\widehat{M}}\subset \widehat{M}$). Denoting by $X^u(A^M)$ the group of unitary unramified characters of $A^M$, restriction along the natural morphism $M\to A^M$ induces a finite quotient morphism $X^u(A^M)\to X^u(M)$ and we choose the unique Haar measure on $X^u(M)$ such that this map is locally measure-preserving when $X^u(A^M)$ is equipped with its Haar measure of total mass $1$.

It is known that the Plancherel measure is absolutely continuous with respect to the ``standard'' measure $d\pi$ fixed as above, that is there exists a function $\mu_G:\Temp(G)\to \bR_{\geq 0}$ such that
$$\displaystyle d\mu_G(\pi)=\mu_G(\pi)d\pi.$$
Moreover, for $\pi=\rI_P^G(\sigma)$ where $\sigma$ is square-integrable, this Plancherel density $\mu_G(\pi)$ further decomposes as a product
$$\displaystyle \mu_G(\pi)=j(\sigma)^{-1} d_\psi(\sigma)$$
where $d_\psi(\sigma)$ denotes as before the formal degree of $\sigma$ and $j(\sigma)$ is a certain product of standard intertwining operators. The conjectural normalization \cite[Appendix II]{LangEis} of such intertwining operators together with Conjecture \ref{Conj HII} then leads to the following prediction.

\begin{conj}[Hiraga-Ichino-Ikeda, Langlands]\label{Conj LHII}
For almost all $\pi\in \Temp(G)$, we have
\begin{equation}\label{eq intro Planch measure}
\displaystyle \mu_G(\pi)=\frac{\deg (\rho_{\pi,\varpi})}{\lvert S_{\phi_\pi}\rvert} \lvert \gamma^*(0,\pi,\Ad_G,\psi)\rvert,
\end{equation}
where, if $n_\pi$ denotes the order of the zero of $\gamma(s,\pi,\Ad_G,\psi)$ at $s=0$,
\begin{equation*}
\displaystyle \gamma^*(0,\pi,\Ad_G,\psi):=\left(\zeta_F(s)^{n_\pi} \gamma(s,\pi,\Ad_G,\psi)\right)_{s=0},
\end{equation*}
with $\zeta_F(s)$ the local zeta factor of $F$.
\end{conj}

\begin{rem}
Identity \eqref{eq intro Planch measure} should really only hold for almost all tempered representations. Indeed, for $P=MU$ a parabolic subgroup and $\sigma$ a square-integrable representation of $M$ it is known that $\lambda\in X^u(M)\mapsto \mu_G(\rI_P^G(\sigma_\lambda))$ extends to a continuous function whereas the right-hand side of \eqref{eq intro Planch measure} isn't continuous at the points $\lambda$ where the order $n_{\rI_P^G(\sigma_\lambda)}$ of the zero drops. The formula should however hold for $\pi=\rI_P^G(\sigma)$ as soon as the normalizer of $(M,\sigma)$ in $G$ is reduced to $M$.
\end{rem}

\subsection{The main result}

Let now $G$ be a symplectic or even special orthogonal group. For such $G$ the local Langlands correspondence has been constructed by Arthur \cite{artbook} (when $G$ is quasi-split) and M\oe{}glin \cite{Moegpaquetstable} with one caveat: when $G$ is an even special orthogonal group, the $L$-parameters are considered up to the action of the outer automorphism $\theta$ coming from conjugation by the nonneutral component of the corresponding full orthogonal group, the $L$-packets are correspondingly $\theta$-stable and the map $\pi\in \Pi^G(\phi)\mapsto \rho_{\pi,\varpi}$ is not necessarily injective. However, for any $L$-parameter of $G$ the Weil-Deligne representations $\Ad_G\circ \phi$ and $\Ad_G\circ \phi^\theta$ (where $\phi^\theta$ is the image of $\phi$ by the automorphism $\theta$) are isomorphic, so that the adjoint $\gamma$-factor of an $L$-parameter for $G$ only depends on its $\theta$-orbit which therefore allows to formulate the Hiraga-Ichino-Ikeda conjecture despite this ambiguity. Moreover, in both cases the centralizers $S_{\phi_\pi}$ are always $2$-abelian groups and therefore the numerators in Conjectures \ref{Conj HII} and \ref{Conj LHII} are both equal to one. Thus, the following theorem gives Conjecture \ref{Conj LHII} for $G$.

\begin{theo}\label{theo1 intro}
Assume that $G$ is either a symplectic of an even special orthogonal group. Then, for almost every $\pi\in \Temp(G)$, we have
\begin{equation*}
\displaystyle \mu_G(\pi)=\frac{1}{\lvert S_{\phi_\pi}\rvert}\lvert \gamma^*(0,\pi,\Ad_G,\psi)\rvert.
\end{equation*}
\end{theo}

\subsection{Spectral decomposition of certain twisted orbital integrals}

To establish \eqref{theo1 intro}, we use twisted endoscopic character relations for $G$, characterizing the tempered $L$-packets $\Pi^G(\phi)$ in terms of twisted characters of some $M=\GL_d(F)$, or more conveniently characters of the twisted space (in the sense of Labesse) $\tM$ of nondegenerate bilinear forms on $F^d$. More precisely, the Dirac measure at a particular central element $\epsilon\in Z(G)$ transfers to a certain linear combination of twisted orbital integrals (see \ref{Sect transfer dirac} for a precise statement)
$$\displaystyle f\in C_c^\infty(\widetilde{M})\mapsto I_{\chi}(f)=\sum_{t\in F^\times/F^{\times,2}} \chi(-t) O(\gamma_t,f),$$
where $\gamma_t\in \tM$ denotes any nondegenerate bilinear form that is the sum of an alternating form and a quadratic form equivalent to $x\mapsto tx^2$ (such element exists and is unique up to $M$-conjugacy). Here $\chi$ is the quadratic character naturally associated to the (pure inner class of) special orthogonal or symplectic group $G$ (e.g. if $G$ is split, for example symplectic, then $\chi$ is trivial). Combining this with the twisted endoscopic relations, Theorem \ref{theo1 intro} is readily reduced to a certain formula giving a sort of ``spectral expansion'' for $I_{\chi}(f)$ in terms of twisted characters.

The main part of this paper is then devoted to proved this spectral expansion of $I_{\chi}$. We refer the reader to Theorem \ref{theo spectraldec orbint} for a precise statement. Let us just mention here that this formula is obtained by computing in two different ways the residue of a certain intertwining operator on the split group $G=\SO_{2d+1}(F)$, more precisely residue at $s=0$ of a self-normalizing intertwining operator
$$\displaystyle \mathcal{M}(\pi,s): \rI^G_P(\pi\otimes \lvert \det\rvert^s)\to \rI^G_{\overline{P}}(\pi\otimes \lvert \det\rvert^s),$$
where $P\subset G$ is the maximal parabolic subgroup with Levi factor $M=\GL_d(F)$ and $\overline{P}$ the opposite parabolic subgroup. The idea that residue of intertwining operators are related to (singular) twisted orbital integrals is not new and is a very nice observation due to Shahidi \cite{Shaend}. This was moreover already used in the original paper of Hiraga, Ichino and Ikeda \cite{HII} to show their conjecture for stable discrete series of unitary groups (see Section 8 of {\it op. cit.}\ ). Besides, their proof can readily be adapted to deal with stable discrete series of other classical groups. 

Thus, the main innovation of this paper is to allow for a similar treatment of all discrete series; including those which are not stable. To achieve this we consider the residue of $\mathcal{M}(\pi,s)$ when $\pi$ is a ``big'' representation of $M$, i.e. one that is not irreducible nor of finite length. More precisely, we will take $\pi$ to be the regular representation on $C_c^\infty(A\backslash M,\chi)$, the space of test functions with a fixed quadratic central character $\chi$ ($A$ denoting the center of $M$). We can still make sense of the standard intertwining operator $\mathcal{M}(\pi,s)$, given by the usual integral, when $\Re(s)>0$ and revisiting Shahidi's original computation this can be shown to admit a meromorphic extension with a simple pole at $s=0$ whose residue is directly related to the twisted orbital integral $I_\chi(f)$.

To compute the same residue spectrally, we essentially combine two ingredients: the Plancherel formula for $M$ (which allows to ``desintegrate'' the induced representation $\rI_P^G(\pi\otimes \lvert .\rvert^s)$ into parabolic inductions of irreducible representations) and Shahidi's normalization of the standard intertwining operators in terms of $\gamma$-factors. More precisely, this reduces the spectral decomposition to the computation of the residue of a very explicit analytic family of distributions on the tempered sual $\Temp(M)$. This computation has incidentally already been done in \cite{BPPlanch}, in the process of establishing an explicit Plancherel decomposition for the Galois symmetric spaces $\GL_n(F)\backslash \GL_n(E)$ (where $E/F$ is a quadratic extension) and leads to the required result.

\subsection{Remark on other classical groups}

The proof presented here can actually be adapted to deal with other classical groups, that is unitary and odd special orthogonal groups. More precisely, if $H$ is a special orthogonal group $\SO(W)$ of a quadratic space $W$ of odd-dimension $2n+1$, there are similar endoscopic relations between stable characters of tempered $L$-packets for $H$ and characters of $M=\GL_{2n}(F)$ twisted by the automorphism $\theta: g\mapsto {}^t g^{-1}$ (so that the underlying twisted space is the same as above), whereas if $H$ is an unitary group $U(W)$ of a Hermitian space $W$ of dimension $n$ over a quadratic extension $E/F$, the endoscopic relations involve characters of $M=\GL_n(E)$ twisted by the automorphism $\theta: g\mapsto {}^t \overline{g}^{-1}$ (where $g\mapsto \overline{g}$ denotes the action of the nontrivial Galois automorphism of $E/F$). In both cases, Conjecture \ref{Conj LHII} can be reduced to a spectral expression for the twisted orbital integral $O(\gamma,f)$, where $\gamma\in \tM=M\theta$ is an element preserving a pinning of $M$, involving exterior square or Asai local gamma factors. Such formula can in turn be obtained by the same strategy, essentially replacing the group $G$ by either the split special orthogonal group $G=\SO_{4n}(F)$ or the quasi-split unitary group $G=U_{2n}(E/F)$ and $P$ by the maximal parabolic subgroup with Levi factor $M$.

However, the formal degree conjecture has already been established in these cases by \cite{ILMfd} and \cite{BPPlanch} respectively and to avoid burdening the reader with lot of case specific computations, I have decided to exclude these groups from the paper.

Besides, as a side remark, the spectral expansion of $O(\gamma,f)$ in terms of the Plancherel density of $H$ was written previously in the work of J. Cohen \cite{CohenJ} as a consequence of the endoscopic character relations. This is however not the same as the spectral formula that is needed to establish the formal degree conjecture in the spirit of the present work, which as explained in the previous paragraph has to be extracted from the Plancherel formula for $M$ more directly and only a posteriori compared with the Plancherel for $H$ via the endoscopic relations.

\subsection{Acknowledgements}

This paper's main result has been presented at numerous seminars and conferences since summer 2021. I would like to apologise to anyone who was promised an initial draft of this paper earlier. I would also like to thank Michael Harris, Wee Teck Gan and Atsushi Ichino warmly for inspiring discussions during the semester-long visit to Columbia University as well as at the summer school 'Arithmetic Geometry in Carthage' in 2019, where most of the details of this paper were envisioned.

This work was funded by the European Union ERC Consolidator Grant, RELANTRA, project number 101044930. Views and opinions
expressed are however those of the authors only and do not necessarily reflect those of the European Union or the European Research Council. Neither the European Union nor the granting authority can be held responsible for them.

	\section{Groups, measures, representations}
	
	\begin{paragr}[Field.]
		We will work over a local non-Archimedean field $F$ of characteristic zero. We denote by $\cO_F$ the ring of integers of $F$, by $\mathfrak{p}_F\subset \cO_F$ the maximal ideal, by $k=\cO_F/\mathfrak{p}_F$ the residual field, $q=\lvert k\rvert$ its cardinality and by $\lvert .\rvert: F\to \bR_{>0}$ the normalised absolute value (i.e. the one such that $\lvert F^\times\rvert=q^{\bZ}$). We denote by $\zeta_F(s)=(1-q^{-s})^{-1}$ the local zeta factor of $F$.
		
		 We fix from now on a nontrivial character $\psi: F\to \bC^\times$ and we denote by $n(\psi)\in \bZ$ the {\em level of $\psi$} that is the largest integer such that $\psi$ is trivial on $\mathfrak{p}_F^{n(\psi)}$. We will always abuse notation and denote the same way a linear algebraic groups defined over $F$ and their groups of $F$-points. Moreover, unless otherwise specified, all the groups and subgroups are defined over $F$.
	\end{paragr}
	
	\begin{paragr}[Haar measures.]\label{S measures}
		For $G$ a connected reductive group over $F$, we denote by $A_G$ the maximal split torus in its center and by $X^*(G)$ its group of algebraic characters (defined over $F$). Using $\psi$, we can equip both $G$ and $A_G$ with canonical Haar measures $dg$ and $da$ as in \cite[\S 2.5]{BPPlanch}. When we want to emphasize the dependance on $\psi$, we will call these measures the {\em $\psi$-Haar measures}. 
		
		Denoting by $A_G^1$ the maximal compact subgroup of $A_G$, we have
		$$\displaystyle \vol(A_G^1,da)=\left(q^{n(\psi)/2}(1-q^{-1})\right)^{\dim(A_G)}=\gamma^*(\mathbf{1}_F,\psi)^{-\dim(A_G)}$$
		where
		$$\displaystyle \gamma(s,\mathbf{1}_F,\psi)=\epsilon(s,\mathbf{1}_F,\psi)\frac{\zeta_F(1-s)}{\zeta_F(s)}$$
		denotes $\gamma$-factor of the trivial character of $F^\times$ and
		$$\displaystyle \gamma^*(\mathbf{1}_F,\psi):=\left(\zeta_F(s) \gamma(s,\mathbf{1}_F,\psi)\right)_{s=0}=\epsilon(0,\mathbf{1}_F,\psi)(1-q^{-1})^{-1}$$
		is its regularized value at $0$.
		
		We set
		$$\displaystyle \cA_G^*=X^*(A_G)\otimes \bR=X^*(G)\otimes \bR.$$
		There is a natural surjective morphism from the complexification $\cA_{G,\bC}^*$ to the group $X^{ur}(G)$ of unramified characters of $G$ sending $\lambda=\chi\otimes s$ to the character $g\mapsto g^\lambda:=\lvert \chi(g)\rvert^s$. Note that elements of the lattice $\frac{2i\pi}{\log(q)}X^*(G)$ induce the trivial character. We equip $i\cA_G^*$ with the unique Haar measure giving the quotient $i\cA_G^*/\frac{2i\pi}{\log(q)}X^*(G)$ volume $1$.
		
		For every Levi subgroup $M\subset G$, the restriction map $X^*(G)\to X^*(M)$ induces a natural embedding $\cA_G^*\subset \cA_M^*$ and we set $(\cA_M^G)^*=\cA_M^*/\cA_G^*$. We equip $i(\cA_M^G)^*$ with the quotient of the Haar measures on $i\cA_M^*$ and $i\cA_G^*$. Thus, by definition, for this measure the lattice $\frac{2i\pi}{\log(q)} X^*(M)/X^*(G)$ has covolume $1$.
		
		The restriction map $X^*(A_M)\to X^*(A_G)$ also induces a surjection $\cA_M^*\to \cA_G^*$ splitting the previous injection. Therefore, we have a canonical decomposition
		$$\displaystyle \cA_M^*=\cA_G^*\oplus (\cA_M^G)^*.$$

		When convenient, we will also assume that non-reductive groups are equipped with arbitrary Haar measures.
	\end{paragr}
	
	\begin{paragr}[Representations.]
		A {\em representation} of $G$ will always mean a smooth representation on a complex vector space. We will make the slight abuse of notation of denoting by the same letter a representation and the space on which it is realized. For a representation $\pi$ of $G$, we write $\pi^\vee$ for its (smooth) contragredient and, for any $\lambda\in \cA_G^*$, by $\pi_\lambda$ the twist of $\pi$ by the character $g\mapsto g^\lambda$. If moreover $\pi$ is irreducible, we denote by $\omega_\pi: A_G\to \bC^\times$ its central character.
		
		We will write $\Irr(G)$ for the set of isomorphism classes of irreducible representations of $G$.
		
		If $\pi$ is a representation of $G$ and $\theta$ an automorphism of $G$, we denote by $\pi^\theta$ the representation given by $\pi^\theta(g):=\pi(\theta(g))$, $g\in G$.
		
		When $P$ is a parabolic subgroup of $G$ with Levi decomposition $P=LN$ and $\sigma$ a representation of $L$, we denote by $\rI_P^G(\sigma)$ the normalised parabolic induction that is the space of smooth functions $e: G\to \sigma$ that satisfy $e(\ell u g)=\delta_P(\ell)^{1/2}\sigma(\ell)e(g)$ for every $(\ell,u,g)\in L\times N \times G$ and where $\delta_P$ denotes the modular character of $P$. If moreover $\sigma$ is unitarizable (e.g. tempered) any invariant scalar product $(.,.)$ induces one on the induction $\rI_P^G(\sigma)$ by the formula
		$$\displaystyle (e,e')=\int_{P\backslash G} (e(g),e'(g))dg,\;\; e,e'\in \rI_P^G(\sigma).$$
	\end{paragr}
	
	\begin{paragr}[Characters.]
	Let $C_c^\infty(G)$ be the space of functions $f:G\to \bC$ that are locally constant and of compact support. Let $\pi$ be an admissible representation of $G$ (for example a representation of finite length). For every $f\in C_c^\infty(G)$, we define an operator $\pi(f)$ on (the space of) $\pi$ by
	\begin{equation*}
	\displaystyle \langle \pi(f)v,v^\vee\rangle=\int_G f(g) \langle \pi(g)v,v^\vee\rangle dg,\;\; v\in \pi,v^\vee\in \pi^\vee.
	\end{equation*}
This operator is of finite rank and we set
$$\displaystyle \Theta_\pi(f)=\Tr \pi(f).$$
The map $f\mapsto \Theta_\pi(f)$ is the {\em distribution-character} of $\pi$.

Assume now that $\pi$ has a central character $\chi: A_G\to \bC^\times$. Let $C_c^\infty(G/A_G,\chi^{-1})$ be the space of locally constant functions $f:G\to \bC$ satisfying $f(ag)=\chi(a)^{-1}f(g)$ for every $(a,g)\in A_G\times G$ and whose support has compact image in $G/A_G$. Then, we can define as above an operator $\pi(f)$ for every $f\in C_c^\infty(G/A_G,\chi^{-1})$ (by simply replacing the integral over $G$ by an integral over $G/A_G$) and we again set $\Theta_\pi(f)=\Tr \pi(f)$.
	\end{paragr}
	
	\begin{paragr}[Discrete series and formal degrees.]
		We denote by $\Pi_2(G)$ the set of isomorphism classes of square-integrable representations of $G$, that is irreducible representations $\pi$ of $G$ such that $\omega_\pi$ is unitary and for every $v,w\in \pi$, $v^\vee,w^\vee\in \pi^\vee$ the integral
		$$\displaystyle \int_{A_G\backslash G} \langle \pi(g)v,v^\vee\rangle \langle w,\pi^\vee(g)w^\vee\rangle dg$$
		converges. Then, there exists $d(\pi)\in \bR_{>0}$ (the {\em formal degree} of $\pi$) such that the above expression equals
		$$\displaystyle \frac{\langle v,w^\vee\rangle \langle w,v^\vee\rangle}{d(\pi)}$$
		for every $v,w,v^\vee,w^\vee$. 
		
	\end{paragr}
	
	\begin{paragr}[Parabolic inductions of discrete series.]
		Let $M$ be a Levi subgroup of $G$. We set
		$$\displaystyle W(G,M)=\Norm_G(M)/M$$
		where $\Norm_G(M)$ stands for the normalizer of $M$ in $G$. Note that $W(G,M)$ acts by automorphisms on $M$.
		
		For $\sigma\in \Pi_2(M)$, we denote by $\rI_M^G(\sigma)$ the isomorphism class of the normalised parabolic induction $\rI_P^G(\sigma)$ where $P$ is a parabolic subgroup with Levi factor $M$ (this does not depend on the choice of $P$) and we let $W(G,\sigma)$ be the subgroup of those $w\in W(G,M)$ such that $\sigma^w\simeq \sigma$.
		
	Let $\Temp_{\ind}(G)$ be the set of $G$-conjugacy classes of pairs $(M,\sigma)$ where $M$ is a Levi subgroup and $\sigma\in \Pi_2(M)$. As is well-known, two such pairs $(M,\sigma)$, $(L,\tau)$ are $G$-conjugate if and only if $\rI_M^G(\sigma)\simeq \rI_{L}^G(\tau)$. We may thus identify $\Temp_{\ind}(G)$ with a set of isomorphism classes of representations of $G$ and we will henceforth always make this identification. We write $[M,\sigma]\in \Temp_{\ind}(G)$ for the $G$-conjugacy class of the pair $(M,\sigma)$.
	
	For every unitary character $\chi:A_G\to S^1$, we denote by $\Temp_{\ind,\chi}(G)$, the subset of those $\pi\in \Temp_{\ind}(G)$ whose central character equal $\chi$. (Note that, although there might be representations in $\Temp_{\ind}(G)$ that are not irreducible, they nonetheless all admit a central character.)
		
		When $G$ is a general linear group (such as the group denoted by $M$ in Sections \ref{Sect Spectral decomp} and \ref{Sect twisted endoscopy}), parabolic inductions of discrete series are all irreducible so that $\Temp_{\ind}(G)$ and $\Temp_{\ind,\chi}(G)$ consist of irreducible representations. In this case, we will drop the index $\ind$ and denote these sets by $\Temp(G)$ and $\Temp_\chi(G)$.
	\end{paragr}
	
	\begin{paragr}[Spectral measure.]\label{S spectral measure}
		We equip $\Temp_{\ind}(G)$ with the unique topology for which the connected components are the subsets of the form
		$$\displaystyle \cO=\{\rI_M^G(\sigma_\lambda)\mid \lambda\in i\cA_M^* \},$$
		for $M$ a Levi subgroup and $\sigma\in \Pi_2(M)$, and the topology induced on such a component is the quotient topology by the map 
		\begin{equation}\label{map quotientO}
		\displaystyle i\cA_M^*\to \cO, \lambda\mapsto \rI_M^G(\sigma_\lambda).
		\end{equation}
		
		We also denote by $d\pi$ the unique (Radon) measure on $\Temp_{\ind}(G)$ such that for every component $\cO$ as above,
		is locally measure-preserving in the following sense: for every $\varphi\in C_c(i\cA_M^*)$, the map \eqref{map quotientO} (which is locally finite) is locally measure preserving in the sense that the pullback of the restriction of $d\pi$ to $\cO$ by this map  coincides with the measure on $i\cA_M^*$ defined in \S \ref{S measures}.

		This measure can be explicitely described locally as follows: for every $[M,\sigma]\in \Temp_{\ind}(G)$, we can find a small enough $W(G,\sigma)$-invariant neighborhood of zero $\cU\subset i\cA_M^*$ such that for every continous compactly supported function $\varphi$ on $\{\rI_M^G(\sigma_\lambda)\mid \lambda\in \cU \}$, we have
		\begin{equation*}
		\displaystyle \int_{\Temp_{\ind}(G)} \varphi(\pi)d\pi=\frac{1}{\lvert W(G,\sigma)\rvert}\int_{\cU} \varphi(\rI_M^G(\sigma_\lambda)) d\lambda.
		\end{equation*}

		For $\chi:A_G\to S^1$ a unitary character, we define similarly a measure $d_{\chi}\pi$ on $\Temp_{\ind,\chi}(G)$ by requiring that for every $[M,\sigma]\in \Temp_{\ind,\chi}(G)$ the map
		$$\displaystyle i(\cA_M^G)^*\to \Temp_{\ind,\chi}(G), \;\; \lambda\mapsto \rI_M^G(\sigma_\lambda),$$
		is locally measure preserving. (And where the measure on $i(\cA_M^G)^*$ is the one fixed in \S \ref{S measures}).

	\begin{lem}\label{lem spectral measures}
	Let $\Unit(A_G)$ be the group of unitary characters of $A_G$ equipped with the measure $d\chi$ dual to the ($\psi$-)Haar measure on $A_G$. Then, for every $\varphi\in C_c(\Temp_{\ind}(G))$, we have
	\begin{equation*}
	\displaystyle \int_{\Temp_{\ind}(G)} \varphi(\pi)d\pi=\frac{\gamma^*(\mathbf{1}_F,\psi)^{-\dim(A_G)}}{[X^*(A_G):X^*(G)]}\int_{\Unit(A_G)} \int_{\Temp_{\ind,\chi}(G)} \varphi(\pi)\; d_\chi\pi\; d\chi.
	\end{equation*}
	\end{lem}

\begin{proof}
The neutral component of $\Unit(A_G)$ consists of the unitary unramified characters of $A_G$, and can be identified by the map $\lambda\mapsto (a\in A_G\mapsto a^\lambda)$ with the quotient $i\cA_G^*/\frac{2i\pi}{\log(q)}X^*(A_G)$. Since for the $\psi$-Haar measure on $A_G$ the maximal compact subgroup $A_G^1$ is of volume $\gamma^*(\mathbf{1}_F,\psi)^{-\dim(A_G)}$, the dual measure $d\chi$ gives $i\cA_G^*/\frac{2i\pi}{\log(q)}X^*(A_G)$ volume $\gamma^*(\mathbf{1}_F,\psi)^{+\dim(A_G)}$. Therefore, as the measure on $i\cA_G^*$ was chosen to give the lattice $\frac{2i\pi}{\log(q)}X^*(G)$ covolume $1$, we see that the local homeomorphism
$$\displaystyle i\cA_G^*\to \Unit(A_G),$$
has Jacobian $\gamma^*(\mathbf{1}_F,\psi)^{+\dim(A_G)}[X^*(A_G):X^*(G)]$.

From there, the lemma readily follows from the definitions of $d\pi$ and $d_\chi\pi$ since, for every Levi subgroup $M\subset G$, the isomorphism $i\cA_G^*\oplus i(\cA_M^G)^*=i\cA_M^*$ is measure-preserving.
\end{proof}
	\end{paragr}
	
	\begin{paragr}[Plancherel densities.]\label{S Plancherel}
		According to Harish-Chandra \cite{WaldPlanch}, there exists a unique continuous function $\mu_G$ on $\Temp_{\ind}(G)$ (the {\em Plancherel density} of $G$) such that for every $f\in C_c^\infty(G)$ we have
		\begin{equation}\label{Plancherel formula}
			\displaystyle f(1)=\int_{\Temp_{\ind}(G)} \Theta_\pi(f) \mu_G(\pi) d\pi,
		\end{equation}
		where $d\pi$ is the measure introduced in the previous paragraph.
		
		In a similar way, for every unitary character $\chi:A_G\to S^1$, there exists a unique continuous function $\mu_{G,\chi}$ on $\Temp_{\ind,\chi}(G)$ such that for every $f\in C_c^\infty(G/A_G, \chi^{-1})$ we have
		\begin{equation}\label{Plancherel formula2}
			\displaystyle f(1)=\int_{\Temp_{\ind,\chi}(G)} \Theta_\pi(f) \mu_{G,\chi}(\pi) d_\chi\pi.
		\end{equation}

The Plancherel density is related to formal degrees in the following way: for every $\pi\in \Pi_2(G)$ with central character $\chi=\omega_\pi$, we have
\begin{equation}\label{relation Planch density fd}
\displaystyle \mu_{G,\chi}(\pi)=d(\pi).
\end{equation}
Moreover, there is a simple relation between $\mu_G$ and $\mu_{G,\chi}$:
	
	\begin{lem}\label{lem Planch densities}
	For every unitary character $\chi: A_G\to S^1$ and every $\pi\in \Temp_{\ind, \chi}(G)$, we have
	\begin{equation*}
	\displaystyle \mu_{G,\chi}(\pi)=\frac{\gamma^*(\mathbf{1}_F,\psi)^{-\dim(A_G)}}{[X^*(A_G):X^*(G)]}\mu_G(\pi).
	\end{equation*}
	\end{lem}

\begin{proof}
This follows directly from Lemma \ref{lem spectral measures}.
\end{proof}
	\end{paragr}

\begin{paragr}
We say of a function $\Phi: \Temp_{\ind}(G)\to \bC$	that it is $C^\infty$ if for every $[M,\sigma]\in \Temp_{\ind}(G)$ the map
$$\displaystyle i\cA_M^*\to \bC,\;\; \lambda\mapsto \Phi(\rI_M^G(\sigma_\lambda)),$$
is $C^\infty$ in the usual sense. We denote by $C_c^\infty(\Temp_{\ind}(G))$ the space of $C^\infty$ functions $\Temp_{\ind}(G)\to \bC$ that are compactly supported, that is that are supported on finitely many connected components of $\Temp_{\ind}(G)$ (all of which are compact).

For $\chi: A_G\to S^1$ a unitary character we define similarly the space $C_c^\infty(\Temp_{\ind,\chi}(G))$.
\end{paragr}

	\section{Spectral decomposition of certain twisted orbital integrals}\label{Sect Spectral decomp}
	
	\subsection{Linear group and twisted linear group}\label{sect twisted linear group}
	
\begin{paragr}
In this chapter we fix an integer $d\geq 1$ and a $d$-dimensional vector space $V$ over $F$. We denote by $V^*$ the dual vector space. For every subspace $U\subset V$, we denote by $U^\perp$ its orthogonal in $V^*$.

We set $M=\GL(V)$ and let $A$ be the center of $M$. We identify $A$ with $F^\times$ in the usual way; that is by the map $\lambda\in F^\times \mapsto \lambda \Id_V\in A$.
\end{paragr}

\begin{paragr}
By the local Langlands correspondence \cite{HTLLC} \cite{HennLLC}, to every irreducible representation $\pi$ of $M$ is associated a $L$-parameter, that is a representation (unique up to isomorphism)
$$\displaystyle \phi_\pi: W'_F\to \GL_d(\mathbb{C}),$$
where $W'_F=W_F\times \SL_2(\bC)$ is the Weil-Deligne group of $F$, which is continuous, algebraic on $\SL_2(\bC)$ and semi-simple.
\end{paragr}

\begin{paragr}[$\gamma$-factors.]
For every finite dimensional representation $\rho: W'_F\to \GL(E)$ of the Weil-Deligne group, that is continuous semi-simple and algebraic on $\SL_2(\bC)$, we denote by $\epsilon(s,\rho,\psi)$ and $L(s,\rho)$ the corresponding local Artin $\epsilon$- and $L$-factors as normalised e.g. in \cite[\S 2.2]{GrossReeder}. The $\epsilon$-factor $\epsilon(s,\rho,\psi)$ is a monomial in $q^{-s}$ and therefore is regular for every complex value of $s$. The resulting $\gamma$-factor is given by
$$\displaystyle \gamma(s,\rho,\psi)=\epsilon(s,\rho,\psi)\frac{L(1-s,\rho^\vee)}{L(s,\rho)},$$
where $\rho^\vee$ is the contragredient of $\rho$. In particular, for $\mathbf{1}_F$ the trivial representation of $W'_F$, we get
$$\displaystyle \gamma(s,\mathbf{1}_F,\psi)=q^{n(\psi)(1/2-s)}\frac{\zeta_F(1-s)}{\zeta_F(s)},$$
where we recall that $n(\psi)$ is the level of the additive character $\psi$ and $\zeta_F(s)=(1-q^{-s})^{-1}$ is the local zeta factor of the field $F$.

For $r:\GL_d(\bC)\to \GL_N(\bC)$ an algebraic representation and $\pi\in \Irr(G)$, we set
$$\displaystyle \epsilon(s,\pi,r,\psi)=\epsilon(s,r\circ \phi_\pi,\psi),\;\; L(s,\pi,r)=L(s,r\circ \phi_\pi),$$
$$\displaystyle \gamma(s,\pi,r,\psi)=\gamma(s,r\circ \phi_\pi,\psi)=\epsilon(s,\pi,r,\psi)\frac{L(1-s,\pi^\vee,r)}{L(s,\pi,r)}.$$
If $\pi\in \Temp(M)$ is tempered, the poles $s_0$ of the $L$-factor $L(s,\pi,r)$ satisfy $\Re(s_0)\leq 0$. In particular, $\gamma(s,\pi,r,\psi)$ is regular at $s=0$ and its vanishing order at this point is equal to the order of the pole of $L(s,\pi,r)$. Denoting by $n_{\pi,r}\geq 0$ this order of vanishing at $s=0$, we set
$$\displaystyle \gamma^*(\pi,r,\psi):=\left(\zeta_F(s)^{n_{\pi,r}} \gamma(s,\pi,r,\psi)\right)_{s=0}.$$

These definitions can in particular be applied to $r=\Sym^2$ and $\bigwedge^2$ the symmetric square and exterior square of the standard representation respectively.
\end{paragr}

\begin{paragr}[Plancherel densities.]
Let $\Ad_M$ be the adjoint representation of $\GL_d(\bC)$ on $\mathfrak{gl}_d(\bC)$. Then, as in \cite[Proposition 2.132]{BPPlanch} we can deduce from \cite{ShaGLn} and \cite[Theorem 2.1]{ILMfd} the following formula for the Plancherel density of $M$:
\begin{equation}\label{formula Planch GL}
\displaystyle \mu_M(\pi)=\omega_\pi(-1)^{d-1}\gamma^*(\pi,\Ad_{M},\psi),\;\; \mbox{ for almost all } \pi\in \Temp(M).
\end{equation}

Let $\chi: A\to S^1$ be a unitary character and $\Ad_{M/A}$ be the adjoint representations of $\GL_d(\bC)$ on $\mathfrak{sl}_d(\bC)$. By Lemma \ref{lem Planch densities}, \eqref{formula Planch GL} implies the following formula
\begin{equation}\label{formula muMchi}
\displaystyle \mu_{M,\chi}(\pi)=\chi(-1)^{d-1}\frac{\gamma^*(\pi,\Ad_{M/A},\psi)}{d}, \; \mbox{ for almost all } \pi\in \Temp_\chi(M).
\end{equation}
\end{paragr}

\begin{paragr}
Let $\Bil(V)$ be the vector space of all bilinear forms $B:V\times V\to F$. There are two commuting left and right actions of $M$ on $\Bil(V)$ given by $(mBm')(x,y)=B(m'x,m^{-1}y)$ for $B\in \Bil(V)$, $x,y\in V$, $m,m'\in M$. This in particular induces the conjugation action $(m,B)\in M\times \Bil(V)\mapsto mBm^{-1}$ that will also be denoted by $(m,B)\mapsto \Ad(m)B$.
	
We let $\Alt(V)$ and $\Sym(V)$ be the subspaces of $\Bil(V)$ consisting of alternating and symmetric forms respectively. Then, any bilinear form $B\in \Bil(V)$ admits a unique decomposition $2B=B^s+B^a$ with $B^s\in \Sym(V)$ and $B^a\in \Alt(V)$. Explicitely:
$$\displaystyle B^s(x,y)=B(x,y)+B(y,x),\; B^a(x,y)=B(x,y)-B(y,x),\;\; \mbox{ for } x,y\in V.$$

We denote by $\tM:=\Bil^*(V)\subset \Bil(V)$ the Zariski open subset of nondegenerate bilinear forms. Then, both the left and the right actions make $\tM$ into a $M$-torsor i.e. the pair $(M,\tM)$ is a {\em twisted space} in the sense of Labesse \cite[Chapitre 2]{LabWal}.
\end{paragr}
	
\subsection{Representations of orthogonal and symplectic types}\label{Sect ortho reps}

\begin{paragr}
We say that an irreducible representation $\pi$ of $M$ is of {\em orthogonal type} if, up to conjugacy, its Langlands parameter $\phi_\pi$ factors through the orthogonal group $\rO_d(\bC)$. The notion of irreducible representation of {\em symplectic type} is defined similarly: $\pi$ is of symplectic type if $d$ is even and, up to conjugacy, $\phi_\pi$ factors through the symplectic group $\Sp_d(\bC)$. Clearly, $\pi$ is of orthogonal (resp. symplectic type) if and only if there exists a non-degenerate symme tric (resp. alternating) form on $\bC^d$ that is fixed by $\phi_\pi$.
\end{paragr}

\begin{paragr}\label{S def Spi+}
Let $\pi\in \Irr(M)$ be of orthogonal type and choose a representative $\phi_\pi$ of its $L$-parameter that factors through $\rO_d(\bC)$. We let
$$\displaystyle S_\pi^{+}=\pi_0(Z_{\rO_d(\bC)}(\phi_\pi))$$
be the component group of the centralizer of $\phi_\pi$ in $\rO_d(\bC)$. It is a $2$-torsion group whose isomorphism class does not depend on the choice of $\phi_\pi$ (inside its conjugacy class).
\end{paragr}

\begin{paragr}
A square-integrable representation $\pi\in \Pi_2(M)$ is of orthogonal type (resp. symplectic type) if and only if the $L$-factor $L(s,\pi,\Sym^2)$ (resp. $L(s,\pi,\bigwedge^2)$) has a pole at $s=0$ or, equivalently, if the gamma factor $\gamma(s,\pi,\Sym^2,\psi)$ (resp. $\gamma(s,\pi,\bigwedge^2,\psi)$) vanishes at $s=0$. Moreover, these two conditions are exclusive of each other: a square-integrable representation cannot be at the same time of orthogonal and symplectic type.
\end{paragr}

\begin{paragr}
Let $\chi: A=F^\times\to \{\pm 1 \}$ be a quadratic character. We denote by $\Temp^{\orth}(M)$ (resp. $\Temp_\chi^{\orth}(M)$) the subset of isomorphism classes of $\pi\in \Temp(M)$ (resp. $\pi\in \Temp_\chi(M)$) that are of orthogonal type. We also set $\Pi_2^{\orth}(M)=\Pi_2(M)\cap \Temp^{\orth}(M)$.

Let $L\subset M$ be a Levi subgroup and $w\in W(M,L)$ be an involution. Then, there exists a representative $\dot{w}\in \Norm_M(L)$ as well as a decomposition
$$\displaystyle L\simeq \prod_{i\in I} (\GL_{n_i}(F)\times \GL_{n_i}(F))\times \prod_{j\in J} \GL_{m_j}(F)$$
through which the action of $\Ad(\dot{w})$ is given by
$$\displaystyle \left((g_i^1,g_i^2)_{i\in I},(g_j)_{j\in J} \right)\mapsto \left((g_i^2,g_i^1)_{i\in I},(g_j)_{j\in J} \right).$$
We let $\Pi_2^{\orth,w}(L)$ be the subset of those $\sigma\in \Pi_2(L)$ that can be decomposed as tensor products
$$\displaystyle \sigma\simeq \bigboxtimes_{i\in I} \rho_i\boxtimes \rho_i^\vee \boxtimes \bigboxtimes_{j\in J} \tau_j$$
where $\rho_i\in \Pi_2(\GL_{n_i})$ for every $i\in I$, $\tau_j\in \Pi_2^{\orth}(\GL_{m_j})$ for every $j\in J$ and moreover the $\tau_j$ are two-by-two distinct. Note in particular that for every $\sigma\in \Pi_2^{\orth,w}(L)$ we have $\sigma^w\simeq \sigma^\vee$. For every $\sigma\in \Pi_2^{\orth,w}(L)$ the induced representation $\pi:=\rI_L^M(\sigma)$ belongs to $\Temp^{\orth}(M)$ and we have an isomorphism
$$\displaystyle S_\pi^{+}\simeq (\bZ/2\bZ)^J.$$
This construction induces a bijection between $\Temp^{\orth}(M)$ and the set of $M$-conjugacy classes of triples $(L,w,\sigma)$, where $L\subset M$ is a Levi subgroup, $w\in W(M,L)$ is an involution and $\sigma\in \Pi_2^{\orth,w}(L)$.
\end{paragr}

\begin{paragr}[Spectral measure.]\label{S orth spectral measure}
Let $L\subset M$ be a Levi subgroup and $w\in W(M,L)$ be an involution as before. We denote by $\cA_{L,w}^*$ the subspace of $\cA_L^*$ on which $\Ad(w)$ acts by $-1$. Note that $\Pi_2^{\orth,w}(L)$ is stable by unramified twists from $i\cA_{L,w}^*$. Let $\sigma\in \Pi_2^{\orth,w}(L)$, $W(M,\sigma)$ be the stabilizer of the isomorphism class of $\sigma$ in $W(M,L)$ and $W(M,\sigma)_w$ be the centralizer of $w$ in $W(M,\sigma)$. (Note that, as $\sigma^w\simeq \sigma^\vee$, $w$ normalizes $W(M,\sigma)$.) Then, for any sufficiently small $W(M,\sigma)_w$-invariant open neighborhood $\cU$ of $0$ in $i\cA_{L,w}^*$, the map $\lambda\mapsto \rI_L^M(\sigma_\lambda)$ induces a bijection between $\cU/W(M,\sigma)_w$ and a subset of $\Temp^{\orth}(M)$. We equip $\Temp^{\orth}(M)$ with the unique topology such that these bijections are actually local homeomorphisms (for every triple $(L,w,\sigma)$ as before). Let us endow $\cA_{L,w}^*$ with the unique Haar measure giving $\cA_{L,w}^*/\frac{2i\pi}{\log(q)}X^*(L)_w$ volume $1$, where we have set $X^*(L)_w=X^*(L)\cap \cA_{L,w}^*$, and $\cU/W(M,\sigma)_w$ with the pushforward of this measure divided by $\lvert W(M,\sigma)_w\rvert$. Then, there exists a unique Radon measure $d^{\orth}\pi$ on $\Temp^{\orth}(M)$ such that the previous embeddings $\cU/W(M,\sigma)_w\hookrightarrow \Temp^{\orth}(M)$ are all measure preserving. We equip $\Temp^{\orth}_\chi(M)$ (which is an open-closed subset of $\Temp^{\orth}(M)$) with the restriction of this measure.

The following proposition, whose proof is essentially a variation of that of \cite[Proposition 3.41]{BPPlanch}, will play a crucial role in this paper but we prefer to postpone its rather computational proof to the appendix.

\begin{prop}\label{prop1 spectral limit}
	Let $\Phi\in C_c^\infty(\Temp_\chi(M))$. Then, we have
	\[\begin{aligned}
		\displaystyle & \lim\limits_{s\to 0^+} d\gamma(s,\mathbf{1}_F,\psi)\int_{\Temp_\chi(M)} \Phi(\pi) \gamma(s,\pi,\Sym^2,\psi)^{-1} \mu_{M,\chi}(\pi)d_\chi\pi= \\
		& 2\chi(-1)^{d-1}\int_{\Temp^{\orth}_\chi(M)} \Phi(\pi) \frac{\gamma^*(0,\pi,\wedge^2,\psi)}{\lvert S_\pi^{+}\rvert} d^{\orth}\pi.
	\end{aligned}\]
\end{prop}

\begin{rem}\label{rmk}
We will actually need to apply Proposition \ref{prop1 spectral limit} to a family of functions $\Phi_s\in C_c^\infty(\Temp_\chi(M))$ that also varies with $s$ in a suitably continuous way. This slightly enhanced form of the proposition can actually be deduced directly from it by the uniform boundedness (aka Banach-Steinhaus) theorem. More specifically, the space $C_c^\infty(\Temp_\chi(M))$ carries a natural topology of locally Fr\'echet space (actually a direct sum of such); it is the direct sum over connected components $\cO$ of $\Temp_{\chi}(M)$ of spaces $C^\infty(\cO)$ that can be each realized as a closed subspace of $C^\infty(i\cA_L^*/\frac{2i\pi}{\log(q)} X^*(L))$ for some Levi $L\subset M$ and the later space has a natural structure of Fr\'echet space ($i\cA_L^*/\frac{2i\pi}{\log(q)} X^*(L)$ being a compact torus). Moreover, for every $s\in \bC$ with $\Re(s)>0$, the linear form
$$\Phi\in C_c^\infty(\Temp_{\chi}(M))\mapsto \int_{\Temp_\chi(M)} \Phi(\pi) \gamma(s,\pi,\Sym^2,\psi)^{-1} \mu_{M,\chi}(\pi)d_\chi\pi$$
is readily seen to be continuous for that topology. Therefore, by the uniform boundedness principle, if $s\mapsto \Phi_s\in C_c^\infty(\Temp_\chi(M))$ is a continuous map defined in a neighborhood of $0\in \bC$, the proposition still holds with $\Phi$ replaced by $\Phi_s$, that is:
\[\begin{aligned}
	\displaystyle & \lim\limits_{s\to 0^+} d\gamma(s,\mathbf{1}_F,\psi)\int_{\Temp_\chi(M)} \Phi_s(\pi) \gamma(s,\pi,\Sym^2,\psi)^{-1} \mu_{M,\chi}(\pi)d_\chi\pi= \\
	& 2\chi(-1)^{d-1}\int_{\Temp^{\orth}_\chi(M)} \Phi_0(\pi) \frac{\gamma^*(0,\pi,\wedge^2,\psi)}{\lvert S_\pi^{+}\rvert} d^{\orth}\pi.
\end{aligned}\]
Finally, we note here a simple criterion to check that a map $s\in U\mapsto \Phi_s\in \Temp_\chi(M)$, where $U\subset \bC$ is open, is continuous: it suffices that the map satisfies the two following conditions
\begin{itemize}
	\item there exists a finite set of connected components $\cO_1,\ldots,\cO_k\subset \Temp_\chi(M)$ such that $\Phi_s(\pi)=0$, for every $s\in U$ and $\pi\in \Temp_\chi(M)\setminus (\cO_1\cup\ldots\cup \cO_k)$;
	
	\item for every pair $(L,\sigma)$, with $L\subset M$ a Levi subgroup and $\sigma\in \Pi_2(L)$, whose central character is equal $\chi$ on $A$, the function
	$$\displaystyle U\times i(\cA_L^M)^*\to \bC,\;\; (s,\lambda)\mapsto \Phi_s(\rI_L^M(\sigma_\lambda)),$$
	is $C^\infty$.
\end{itemize} 
\end{rem}
\end{paragr}

\begin{paragr}
Every $\pi\in \Temp^{\orth}(M)$ is self-dual and can thus be extended to a {\em unitary twisted representation} $\tpi$ of $\tM$ i.e. a map
$$\displaystyle \tpi: \tM\to \GL(V_\pi),$$
where $V_\pi$ denotes the space of the representation $\pi$, satisfying
$$\displaystyle \tpi(g\tm g')=\pi(g)\tpi(\tm)\pi(g'), \mbox{ for every } (g,\tm,g')\in M\times \tM\times M,$$
and whose image preserves an inner product on $V_\pi$. Such an extension is unique up to multiplication by a scalar $z\in \mathbb{S}^1$ of module $1$. We will fix such an extension for every $\pi\in \Temp^{\orth}(M)$ henceforth.
\end{paragr}

\subsection{On certain twisted conjugacy classes}	
	
	For every $t\in F^\times$, we choose an element $\gamma_t\in \tM$ such that the quadratic form $\gamma_t^s$ is of rank $1$ and isomorphic to $x\mapsto tx^2$. We also fix an alternating form $\gamma_0\in \Alt(V)$ of maximal rank, that is of rank $d$, if $d$ is even, and of rank $d-1$, if $d$ is odd. In particular, note that $\gamma_0\in \tM$ if and only if $d$ is even.

	\begin{lem}
		\begin{enumerate}[(i)]
			\item The alternating form $\gamma_t^a$ is of maximal rank, i.e.\ it is $M$-conjugate to $\gamma_0$;
			
			\item The $M$-conjugacy class of $\gamma_t$ only depends on the square class of $t$, that is of its image in $F^\times/F^{\times,2}$;
			
			\item The conjugacy class $\Ad(M)\gamma_{t}$ of $\gamma_t$ carries a (unique up to a scalar) $\Ad(M)$-invariant measure and for every $a\in A=F^\times$, the map $\tilde{m}\mapsto \tilde{m}a^2$ induces a measure-preserving self-homeomorphism of $\Ad(M)\gamma_t$;
			
			\item The closure (for the analytic topology)  $\tM^\sharp$ of $\Ad(M)\gamma_{-1}$ in $\tM$ satisfies
			$$\displaystyle \tM^\sharp=\left\{\begin{array}{ll}
				\Ad(M)\gamma_{-1} & \mbox{ if } d \mbox{ is odd}; \\
				\Ad(M)\gamma_{-1}\sqcup \Ad(M)\gamma_0 & \mbox{ if } d \mbox{ is even.}
			\end{array} \right.$$
		\end{enumerate}
	\end{lem}
	
	\begin{proof}
		\begin{enumerate}[(i)]
			\item is clear since alternating forms are of even rank and $\gamma_t$ is non-degenerate.
			
			\item Since $M$ acts transitively on the set of alternating forms of maximal rank, it suffices to check that $M_{\gamma_0}$ (the stabilizer of $\gamma_0$ in $M$) acts transitively on the set of quadratic forms $q\in \Sym(V)$ that are of rank $1$, isomorphic to $x\mapsto tx^2$ and such that $\gamma_0+q$ is nondegenerate. Such a quadratic form is of the form $q(v,v')=t\ell(v)\ell(v')$ where $\ell\in V^*$ is nonzero on the kernel $\Ker(\gamma_0)$ of the alternating form $\gamma_0$, that is $\ell\notin \Ker(\gamma_0)^\perp$. In the case where $d$ is even, $\Ker(\gamma_0)^\perp=V^*$ and $M_{\gamma_0}\subset M$ is a symplectic group which is well-known to act transitively on $V^*-\{0 \}$. When $d$ is odd, $L=\Ker(\gamma_0)$ is a line and $M_{\gamma_0}$ consists of those elements $g\in M$ preserving $L$ as well as the symplectic form induced by $\gamma_0$ on $V/L$. It is then easy to check that this subgroup acts transitively on $V^*-L=V^*- \Ker(\gamma_0)^\perp$.
			
			\item The existence of an $\Ad(M)$-invariant measure on $\Ad(M)\gamma_{t}$ is equivalent to the stabilizer $M_{\gamma_t}$ being unimodular which can either be checked directly in the case at hand but holds more generally for every conjugacy class in a twisted space. The second part of the statement just follows from noting that $a^{-1}\tilde{m}a=\tilde{m}a^2$ for every $a\in A$ and $\tilde{m}\in \tM$.
			
			\item Denote by $\rk(B)$ the rank of a bilinear form $B\in \Bil(V)$, that is the rank of the corresponding linear map $V\to V^*$. We readily check that for every integer $k\geq 0$, the subset of those $B\in \Bil(V)$ with $\rk(B^a)\leq k$ (resp. with $\rk(B^s)\leq k$) is closed. Moreover, the discriminant map
			$$\disc: \tM\to F^\times/F^{\times,2},$$
			sending $\gamma\in \tM$ to the square class of the determinant $\det(\gamma(v_i,v_j))_{1\leq i,j\leq d}$ for any basis $(v_i)_{i=1}^d$ of $V$, is continuous. This already implies that $\Ad(M)\gamma_{-1}$ is closed in $\tM$ when $d$ is odd since it coincides with the subset of those $\gamma\in \tM$ with $\rk(\gamma^s)\leq 1$ and $\disc(\gamma)=(-1)^{(d+1)/2}$. Assume now that $d$ is even. We see similarly that $\Ad(M)\gamma_{-1}\sqcup \Ad(M)\gamma_0$ is closed in $\tM$ as it coincides with the subset of those $\gamma\in \tM$ with $\rk(\gamma^s)\leq 1$ and $\disc(\gamma)=(-1)^{d/2}$. To conclude, it suffices to see that $\tM^\sharp$ contains an alternating form, and therefore the whole of $\Ad(M)\gamma_0$. Up to conjugacy, we may assume that $\gamma_{-1}^a=\gamma_0$. Then, $\gamma_{-1}^s$ is of the form $(v,v')\mapsto -\ell(v)\ell(v')$ where $\ell\in V^*\setminus \Ker(\gamma_0)^{\perp}$. Let $L=\Ker(\gamma_0)$ and $H$ be the kernel of $\ell$. In particular we have the decomposition $V=L\oplus H$ where $L$ is a line and $H$ an hyperplane. For every $\lambda\in F^\times$, let us denote by $a_\lambda$ the unique element of $M=\GL(V)$ that acts by the identity on $H$ and $\lambda$ times the identity on $L$. Then, we have $a_{\lambda}^{-1}\gamma_0a_\lambda=\gamma_0$ and $a_{\lambda}^{-1}\gamma_{-1}^sa_\lambda=\lambda^2\gamma_{-1}^s$ for every $\lambda\in F^\times$. In particular, we have
			$$\lim\limits_{\lambda\to 0} a_{\lambda}^{-1}\gamma_{-1}a_\lambda=\gamma_0\in \tM^\sharp.$$
		\end{enumerate}
	\end{proof}
	
	
	We henceforth fix $\Ad(M)$-invariant measures $d_t \tm$ on $\Ad(M)\gamma_t$ for every $t\in F^\times$ such that for every $a\in F^\times$ the homeomorphism $\Ad(M)\gamma_t\simeq \Ad(M) \gamma_{at}$, $\tm\mapsto \tm a$, preserves the chosen measures. (This is clearly possible by the above lemma.) We denote by
	$$\displaystyle O(\gamma_t,f)=\int_{\Ad(M)\gamma_t} f(\tm) d_t\tm,\;\; f\in C_c^\infty(\tM),\; t\in F^\times/F^{\times,2},$$
	the corresponding orbital integrals. When $d$ is even, the orbit $\Ad(M)\gamma_t$ isn't closed in $\tM$ and thus the convergence of the above orbital integrals is not completely formal, but follows from Rao's convergence theorem for unipotent orbital integrals \cite{Rao}.
	
	For every quadratic character $\chi: F^\times \to \{\pm 1 \}$ and $f\in C_c^\infty(\tM)$, we also set
	\begin{equation}\label{def Ichi}
		\displaystyle \rI_{\chi}(f)=\sum_{t\in F^\times/F^{\times,2}} \chi(-t) O(\gamma_t,\tf).
	\end{equation}

	\subsection{The theorem}

	\begin{theo}\label{theo spectraldec orbint}
		Let $\chi: A=F^\times\to \{\pm 1 \}$ be a quadratic character. Then, there exists an absolute constant $C>0$ (depending, in particular, on the normalization of the invariant measures on the conjugacy classes $\Ad(M)\gamma_t$) as well as a function $\pi\in\Temp^{\orth}_\chi(M)\mapsto c(\tpi)\in \mathbb{S}^1$ (which depends on the choice of the normalizations of the twisted representations $\tpi$) such that
		\begin{equation*}
			\displaystyle \rI_{\chi}(f)=C \int_{\Temp_\chi^{\orth}(M)} \Theta_{\tpi}(f) c(\tpi)\frac{\gamma^*(0,\pi,\wedge^2,\psi)}{\lvert S_\pi^{+}\rvert} d\pi
		\end{equation*}
		for every $f\in C_c^\infty(\tM)$, where we recall that the distribution on the left-hand side is defined by \eqref{def Ichi}.
	\end{theo}

The proof of Theorem \ref{theo spectraldec orbint} will occupy the rest of this section, up to the end of Subsection \ref{IO and spectral exp}.
	
	\subsection{Embedding in an odd special orthogonal group}
	Set
	$$\displaystyle U=V\oplus L\oplus V^*$$
	where $L=F$. We equip $U$ with the quadratic form $Q$ defined by
	$$\displaystyle Q((x,\lambda,x^*),(y,\mu,y^*))=\langle x,y^*\rangle+\langle y,x^*\rangle +\lambda\mu,\;\; (x,\lambda,x^*),(y,\mu,y^*)\in U,$$
	where $\langle .,.\rangle$ stands for the canonical pairing between $V$ and $V^*$.
	
	Let $G=\SO(U,Q)$ be the special orthogonal group of the quadratic space $(U,Q)$. We denote by $P$ and $\overline{P}$ the parabolic subgroups of $G$ stabilizing the maximal isotropic subspaces $V$ and $V^*$ respectively:
	$$\displaystyle P=\Stab_G(V),\;\;\; \overline{P}=\Stab_G(V^*).$$
	Then, $P$, $\overline{P}$ are opposite and their common Levi factor $P\cap \overline{P}$ can be identified, by restriction to $V$, with $M=\GL(V)$. We denote by $N$, $\overline{N}$ the unipotent radicals of $P$ and $\overline{P}$ respectively.
	
	\begin{prop}\label{prop irred generic}
	For $\pi\in \Temp^{\orth}(M)$ in general position (more precisely: outside of a subset of $\Temp^{\orth}(M)$ that is negligible with respect to the measure $d^{\orth}\pi$), the induced representation $\rI_M^G(\pi)$ is irreducible.
	\end{prop}
	
	\begin{proof}
	By the discussion in Subsection \ref{Sect ortho reps}, we can write $\pi=\rI_L^M(\sigma)$ where $L\subset M$ is a Levi subgroup of the form
	$$\displaystyle L= \prod_{i\in I} (\GL_{n_i}\times \GL_{n_i})\times \prod_{j\in J} \GL_{m_j},$$
	and $\sigma$ is a square-integrable representation of $L$ of the form
	$$\displaystyle \sigma=\bigboxtimes_{i\in I} \rho_i\boxtimes \rho_i^\vee \boxtimes \bigboxtimes_{j\in J} \tau_j$$
	where $\rho_i\in \Pi_2(\GL_{n_i})$ for every $i\in I$ and $\tau_j\in \Pi_2^{\orth}(\GL_{m_j})$ for every $j\in J$. Clearly, if $\pi$ is in general position then none of the $\rho_i$, $i\in I$, is of symplectic type. Note that, since the $\tau_j$ are already of orthogonal type, none of them is of symplectic type either. Thus, it suffices to show more generally that if
	$$\displaystyle L=\GL_{n_1}\times \ldots \times \GL_{n_k}$$
	is a Levi subgroup of $M$ and
	$$\displaystyle \sigma=\tau_1\boxtimes \ldots \boxtimes \tau_k\in \Pi_2(L)$$
	is a square-integrable representation of $L$ with none of the $\tau_i$ of symplectic type, then $\rI_M^G(\pi)\simeq \rI_L^G(\sigma)$ is irreducible.
	
	As follows from the theory of the $R$-group, the parabolic induction of a discrete series is irreducible if and only if the associated $R$-group is trivial. The computation of the $R$-group of $\rI_L^G(\sigma)$ has been reduced in \cite{GoldSO} to the case where $L$ is a maximal Levi subgroup. More precisely, from Theorem 4.18 and Theorem 6.5 of {\it loc.\ cit.}\, to deduce the irreducibility of $\rI_L^G(\sigma)$ for $\sigma$ as above, it suffices to check the following (up to replacing the vector space $V$ by a smaller subspace): for $\sigma\in \Pi_2(M)$ that is not of symplectic type, $\rI_M^G(\sigma)$ is irreducible. However, by \cite[Corollary 5.4.2.3]{Silb}, $\rI_M^G(\sigma)$ is reducible if and only if $\sigma\simeq \sigma^\vee$ and the standard intertwining operator
	$$\displaystyle \mathcal{M}(\sigma,s): \rI_P^G(\sigma_s)\to \rI_{\overline{P}}^G(\sigma_s),$$
	where we have set $\sigma_s:=\sigma\otimes \lvert \det\rvert^s$, has no pole at $s=0$. Since $\sigma$ is a discrete series that is not of symplectic type, if $\sigma\simeq \sigma^\vee$ then it is of orthogonal type. Moreover, by the normalization given in \cite{ShaLangconj} of the standard intertwining operators, $\mathcal{M}(\sigma,s)$ has a pole at $s=0$ if and only if the local $L$-factor $L^{\Sh}(s,\sigma,\Sym^2)$ defined by Shahidi has a pole at the same point. By \cite{HenLfn}, we have the identity $L^{\Sh}(s,\sigma,\Sym^2)=L(s,\sigma,\Sym^2)$ and the latter has a pole at $s=0$ when $\sigma$ is of orthogonal type. This proves the claim and hence the proposition.
	\end{proof}

\subsection{A lemma on the Bruhat decomposition of elements in $\overline{N}$}
	
	To any $g\in G$, we associate $B_g\in \Bil(V)$ defined by
	$$B_g(x,y)=Q(gx,y),\;\; x,y\in V.$$
	The map $g\in G\mapsto B_g\in \Bil(V)$ is regular and $M\times M$-equivariant, where the first and second copy of $M$ act by left and right translations on $G$ respectively.
	
	Let $\Norm_G(M)$ be the normalizer of $M$ in $G$. Then, $g\mapsto B_g$ identifies, $M\times M$-equivariantly, the non-neutral component $\Norm_G(M)\setminus M$ with $\tM\subset \Bil(V)$.
	
	For $\overline{u}\in \overline{N}$, the composition of the restriction of $\overline{u}$ to $V$ with the projection onto $L$ (parallel to the subspace $V\oplus V^*$) yields a linear form $\ell_{\overline{u}}\in V^*$. This induces a morphism $\overline{N}\to V^*$ that is readily seen to be surjective. Let $Z(\overline{N})$ be the kernel of this morphism so that we have a short exact sequence
	\begin{equation}\label{eq0}
		\displaystyle 0\to Z(\overline{N})\to \overline{N}\to V^*\to 0.
	\end{equation}
	Then, $Z(\overline{N})$ is included in the center of $\overline{N}$ (and is actually the full center whenever $d\geq 2$) and the map $z\in Z(\overline{N})\mapsto B_z$ yields an isomorphism $Z(\overline{N})\simeq \Alt(V)$.
	
	Set
	$$\displaystyle G'=N \tM N.$$
	Then, $G'$ is the Zariski open subset consisting of those $g\in G$ such that the bilinear form $B_g$ is non-degenerate. Indeed, from the definition it is easy to see that, for $g\in G$, the form $B_g$ is nondegenerate  if and only if $gV\cap (L\oplus V)=0$ or equivalently, since every isotropic subspace of $L\oplus V$ is contained in $V$, $gV\cap V=0$. This last condition is equivalent to $gPg^{-1}$ being a parabolic subgroup opposite to $P$. As $P$ acts transitively on the set of parabolic subgroups opposite to it and is self-normalizing, we deduce that the set of those $g\in G$ with $B_g$ nondegenerate is a $P$ double coset. But $G'$ is certainly such a double coset and it contains elements $g$ with $B_g$ nondegenerate (e.g. those of $\tM$) hence the result.
	
	The regular map $N\times \tM\times N\to G'$, $(u_1,\tm,u_2)\mapsto u_1 \tm u_2$ is an isomorphism. We shall denote by $g\in G'\mapsto \tm(g)\in \tM$ the unique regular map such that $g\in N\tm(g)N$ for every $g\in G'$. This map is $M\times M$-equivariant.
	
	Set $\oN'=\oN\cap G'$. It is a Zariski dense open subset of $\oN$ and the regular map $\ou\in \oN\mapsto \tm(\ou)\in \tM$ is $M$-equivariant for the adjoint actions. Set $\epsilon=-\Id_V\in A$.
	
	The following proposition is a reformulation of part of the discussion of \cite[\S 10]{Shaend} pertaining to $\SO(2n+1)$ (see in particular Lemmas 10.3 and 10.4 of {\em loc.\ cit.}\ ).
	
	\begin{prop}[Shahidi]
		The map $\ou\in \oN'\mapsto \tm(\ou)\in \tM$ is $\Ad(\epsilon)$-invariant and induces an homeomorphism $\oN'/\Ad(\epsilon)\simeq \tM^\sharp$.
	\end{prop}
	
	\begin{proof}
		We use the following elementary facts whose proofs are left to the reader.
		\begin{enumerate}
			\item For every $\ou\in \oN$, we have $B_{\ou}^s(x,y)=-\ell_{\ou}(x)\ell_{\ou}(y)$ for $x,y\in V$.
			
			\item For all $z\in Z(\oN)$, $\ou\in \oN$ we have $B_{z\ou}=B_z+B_{\ou}$.
		\end{enumerate}
		This already shows that the image of $\ou\in \oN\mapsto B_{\ou}$ is the closed subset $\Bil(V)^\sharp$ of those $B\in \Bil(V)$ such that $B^s$ is of rank at most one and equivalent to $x\mapsto -x^2$ if nonzero. Furthermore, it also readily follows from 1., 2. and the short exact sequence \eqref{eq0} that $\ou\mapsto B_{\ou}$ induces an homeomorphism $\oN/\Ad(\epsilon)\simeq  \Bil(V)^\sharp$. The proposition follows by restriction to the open subset of nondegenerate forms.
	\end{proof}
	
	Let $\delta$ be the modular character of $P$. Then, we have $\delta(m)=\lvert \det(m)\rvert^{d}$ for every $m\in M=\GL(V)$. We also fix Haar measures on $N$, $\oN$. For every $\tm\in \tM$, $\Ad(\tm)$ induces an isomorphism $N\simeq \oN$ and we denote by $\widetilde{\delta}(\tm)$ its Jacobian. Note that
	\begin{equation}\label{eq tildedelta}
		\displaystyle \widetilde{\delta}(m_1\tm m_2)=\delta(m_1)^{-1}\widetilde{\delta}(\tm)\delta(m_2),\;\; \mbox{ for every } (m_1,\tm,m_2)\in M\times \tM\times M.
	\end{equation}
	In particular, the measure $\widetilde{\delta}(\tm)^{1/2} d_{-1} \tm$ on $\Ad(M)\gamma_{-1}$ transform according to the modular character $\delta_{\overline{P}}=\delta^{-1}$ of $\overline{P}$ under the adjoint action of $M$.
	
	Let $\oN''\subset \oN'$ be the inverse image of $\Ad(M)\gamma_{-1}$ by the map $\ou\mapsto \tm(\ou)$. By the proposition, $\oN''$ is a Zariski dense open subset of $\oN'$ (and is actually equal to it when $d$ is odd) that is homogeneous under $\Ad(M)$. The homeomorphism $\tm: \oN''/\Ad(\epsilon)\simeq \Ad(M)\gamma_{-1}$ sends the (restriction of) fixed Haar measure on $\oN$ to a positive multiple of the measure $\widetilde{\delta}(\tm)^{1/2} d_{-1} \tm$ on $\Ad(M)\gamma_{-1}$, that is there exists a constant $C_0>0$ such that
	\begin{equation}\label{eq2}
		\displaystyle \int_{\oN'} f(\tm(u)) du= C_0 \cdot \rO(\gamma_{-1},\widetilde{\delta}^{1/2}f)
	\end{equation}
	for every $f\in C_c^\infty(\tM)$.
	
	\subsection{Convenient sections and orbital integrals}
	
	For every character $\omega: A\to \bC^\times$, we denote by $C_c^\infty(AN\backslash G,\omega)$ the space of functions
	$$\Phi:N\backslash G\to \bC$$
	that are right-invariant by a compact-open subgroup, satisfy
	$$f(ag)=\omega(a)f(g)$$
	for $(a,g)\in A\times G$ and whose support has compact image in $AN\backslash G$.
	 
	By a {\em convenient section}, we mean a family of functions
	$$\displaystyle s\in \bC\mapsto \Phi_s\in C_c^\infty(AN\backslash G,\chi \delta^{1/2}\lvert \det \rvert^s)$$
	such that we can find $\Phi\in C_c^\infty(N\backslash G)$ with
	\begin{itemize}
		\item For every $s\in \bC$ and $g\in G$,
		$$\displaystyle \Phi_s(g)=\int_{A} \Phi(ag) \chi(a) \delta(a)^{-1/2} \lvert \det(a)\rvert^{-s}da;$$
		\item $\Supp(\Phi)\subset G'$. 
	\end{itemize}
	In particular, if $s\mapsto \Phi_s$ is a convenient section we have $\Supp(\Phi_s)\subset G'$ for every $s\in \bC$. 
	
	Fix a quadratic character $\chi:A=F^\times\to \{\pm 1 \}$ and let $f \in C_c^\infty(\tM)$. We fix from now on a function $\Phi\in C_c^\infty(N\backslash G')$ such that
	\begin{equation}\label{eq tf}
	\displaystyle f=\widetilde{\delta}^{1/2} \cdot \Phi^\vee\mid_{\tM},
\end{equation}
where $\Phi^\vee\mid_{\tM}$ denotes the restriction of the function $\Phi^\vee: g\mapsto \Phi(g^{-1})$ to $\tM\subset G$, and we consider the convenient section $s\in \bC\mapsto \Phi_s\in C_c^\infty(AN\backslash G,\chi \delta^{1/2}\lvert \det \rvert^s)$ deduced from $\Phi$ as above.
	
	\begin{prop}\label{prop}
		The integral
		\begin{equation}\label{eq4}
			\displaystyle \int_{\oN} \Phi_s(\ou) d\ou
		\end{equation}
		converges for $\Re(s)>0$ and moreover we have the identity
		$$\displaystyle \lim\limits_{s\to 0^+} \gamma(s,\mathbf{1}_F,\psi) \int_{\oN} \Phi_s(\ou)d\ou =C_1 \cdot \rI_{\chi}(f)$$
		where $C_1>0$ is an absolute constant (depending on our choice of Haar measures). 
	\end{prop}
	
	\begin{proof}
		Let $A^2$ denote the subgroup of squares in $A$. We have the integration formula\footnote{Indeed, it suffices to check that it holds for the characteristic function $\phi=\mathbf{1}_{\cO_F^\times}$; note that $2\lvert 2\rvert^{-1}$ is precisely the cardinality of the group of square classes $\cO_F^\times/\cO_F^{\times,2}$.}
		$$\displaystyle \int_A \phi(a)da=\frac{\lvert 2\rvert}{2}\sum_{t\in A/A^2} \int_A \phi(ta^2) da,\;\; \phi\in C_c^\infty(A).$$
		Thus, by definition of the convenient section $s\mapsto \Phi_s$, and using the fact that $\chi$ is quadratic, we have (ignoring for the moment convergence issues)
		\begin{align}\label{eq3}
			\displaystyle \int_{\oN} \Phi_s(\ou) d\ou & =\int_{\oN\times A} \Phi(a\ou) \chi(a)\delta(a)^{-1/2} \lvert a\rvert^{-ds} dad\ou \\
			\nonumber & = \frac{\lvert 2\rvert}{2} \sum_{t\in A/A^2} \delta(t)^{-1/2}\chi(t) \lvert t\rvert^{-ds}\int_{\oN\times A} \Phi(ta^2\ou)\delta(a)^{-1} \lvert a\rvert^{-2ds} dad\ou \\
			\nonumber & = \frac{\lvert 2\rvert}{2}  \sum_{t\in A/A^2} \delta(t)^{-1/2}\chi(t)  \lvert t\rvert^{-ds}\int_{\oN\times A} \Phi^t(a\ou a) \lvert a\rvert^{-2ds} dad\ou
		\end{align}
		where we have set $\Phi^t(g)=\Phi(tg)$ for any $g\in G$ and the last equality follows from the change of variable $\ou\mapsto a^{-1}\ou a$. Fix $t\in A$ and let $u\in \oN'$. There exists a unique $u(\ou)\in N$ such that $\ou\in N \tm(u)u(\ou)$ and we have
		$$\displaystyle \int_A \Phi^t(a\ou a) \lvert a\rvert^{-2ds} da=\int_A \Phi^t(\tm(\ou) a u(\ou)a^{-1}) \lvert a\rvert^{2ds}da.$$
		Note that the function $a\in F^\times \mapsto \Phi^t(\tm(\ou) au(\ou)a^{-1})$ vanishes for $a\in F^\times$ sufficiently large (by compactness of the support of $f$ and because $u(\ou)\neq 1$) and is constant equal to $\Phi^t(\tm(u))$ in a neighborhood of $0$ (by smoothness of $\Phi$) i.e. it extends to a function in $C_c^\infty(F)$. Thus, by Tate's thesis the above integral converges for $\Re(s)>0$ and is equivalent to $\gamma(2ds,\mathbf{1}_F,\psi)^{-1}\Phi^t(\tm(u))$ when $s\to 0$. Moreover, we claim that all of this happens uniformly in $\ou$. More precisely, as $\Supp(\Phi^t)\subset G'=N \tM N$, we may find a compact subset $\Omega\subset \tM$ such that $\Supp(\Phi^t)\subset N \Omega N$. It follows that the function $a\mapsto \Phi^t(a\ou a)$ vanishes identically unless $\ou$ belongs to the compact subset $\Omega'=\{u\in \oN'\mid \tm(u)\in \Omega \}$. Moreover, as the image of the map $\ou\in \Omega'\mapsto u(\ou)\in N$ is also compact (and avoiding $1\in N$), the previous reasoning shows that the functions $a\mapsto \Phi^t(a\ou a)$ ranges over a finite set as $\ou$ varies in $\Omega'$. Thus, we may write the integral over $\oN$ in the last line of \eqref{eq3} as a finite sum to deduce that \eqref{eq4} converges for $\Re(s)>0$ and
		\[\begin{aligned}
			\displaystyle \lim\limits_{s\to 0} \gamma(s,\mathbf{1}_F,\psi)\int_{\oN} \Phi_s(\ou)d\ou & = \frac{\lvert 2\rvert}{4d}\sum_{t\in A/A^2} \chi(t) \delta(t)^{-1/2} \int_{\oN} \Phi^t(\tm(\ou)) d\ou \\
			& = C_0\cdot \frac{\lvert 2\rvert}{4d} \sum_{t\in F^\times/F^{\times,2}} \chi(t) \int_{\Ad(M)\gamma_{-1}} f(t \tm) d_{-1}\tm \\
			& = C_1\cdot \sum_{t\in F^\times/F^{\times,2}} \chi(-t) O(\gamma_t,f)=C_1 \cdot \rI_{\chi}(f),
		\end{aligned}\]
		where we have set $C_1= C_0\frac{\lvert 2\rvert}{4d} $. Here, the second equality above follows from \eqref{eq2} as well as the relations \eqref{eq tildedelta} and \eqref{eq tf}, whereas the last equality is a consequence of the fact that the isomorphism $\Ad(M)\gamma_{-1}\simeq \Ad(M)\gamma_{-t}$, $\tm\mapsto t\tm$, is measure preserving.
	\end{proof}
	
	\subsection{Intertwining operators and the spectral expansion}\label{IO and spectral exp}
	
	We continue with the setting of the previous section (i.e. with the function $\Phi\in C_c^\infty(N\backslash G')$ and the convenient section $s\mapsto \Phi_s$ that it induces). Let $g\in G$. Then, the restriction $(R(g)\Phi_s)\mid_M$ of the translate $R(g)\Phi_s$ of $\Phi_s$ to $M$ belongs to $C_c^\infty(A\backslash M, \chi \delta^{1/2} \lvert \det\rvert^s)$. Applying the Plancherel formula of $M$ to this function, we obtain that
	\begin{equation}\label{eq5}
		\displaystyle \Phi_s(g)=\int_{\Temp_\chi(M)} \Tr \Phi_{\pi,s}(g) \mu_{M,\chi}(\pi)\; d_\chi \pi,\;\mbox{ for all } g\in G \mbox{ and } s\in \bC,
	\end{equation}
	where we have denoted by $\Phi_{\pi,s}(g)$ the operator
	$$\displaystyle \Phi_{\pi,s}(g)=(\pi_s\otimes \delta^{1/2})((R(g)\Phi)^\vee\mid_M),$$
	for each $\pi\in \Temp_\chi(M)$. (We recall that $\pi_{s}:=\pi \otimes \lvert \det\rvert^{s}$ and that for every function $\varphi$, $\varphi^\vee$ denotes the function $\gamma\mapsto \varphi(\gamma^{-1})$.) This operator can be identified with an element of the tensor product $\pi_s\otimes (\pi_s)^\vee=\pi_s\otimes (\pi^\vee)_{-s}$ and, if we let $M$ act only on the first factor by the representation $\pi_s$, the function $g\mapsto \Phi_{\pi,s}(g)$ belongs to the induced representation $\rI_{P}^G(\pi_s\otimes (\pi^\vee)_{-s})$, that is:
	\begin{equation}\label{eq6}
		\displaystyle \Phi_{\pi,s}\in \rI_{P}^G(\pi_s\otimes (\pi_s)^\vee)\simeq \rI_{P}^G(\pi_s)\otimes (\pi_s)^\vee.
	\end{equation}
	
	For $\pi\in \Temp(M)$, let us denote by
	$$\displaystyle \cM(\pi,s): \rI_{P}^G(\pi_s)\to \rI_{\overline{P}}^G(\pi_s), \; s\in \bC,$$
	the standard intertwining operator. For $\Re(s)>0$, $\cM(\pi,s)$ is characterized by the relation
	\begin{equation}\label{int std IO}
	\displaystyle \langle (\cM(\pi,s)\phi)(g),v^\vee\rangle=\int_{\oN} \langle \phi(\ou g), v^\vee\rangle d\ou,\mbox{ for every } \phi\in \rI_{P}^G(\pi_s),\; g\in G, \; v\in (\pi_s)^\vee,
	\end{equation}
	the integral being absolutely convergent, and for general $s$ this operator is defined by rational continuation. Since we will soon need to use a uniform version for it, let us recall the basic estimate showing the convergence of the above integral. Let $K\subset G$ be a special maximal compact subgroup; so that in particular we have an Iwasawa decomposition $G=PK$. Choose, for every $\ou\in \oN$, elements $m(\ou)\in M$ and $k(\ou)\in K$ such that $\ou\in Nm(\ou)k(\ou)$. Then, for $(\ou,g)\in \oN\times G$, we have
	$$\displaystyle \phi(\ou g)=\lvert \det(m(\ou))\rvert^s \delta(m(\ou))^{1/2}\pi(m(\ou))\phi(k(\ou)g).$$
	As $\ou$ varies, the vector $\phi(k(\ou) g)\in V_\pi$ only assumes finitely many values. Thus, the convergence of the right-hand side of \eqref{int std IO} reduces to that of the integral
	$$\displaystyle \int_{\oN}  \lvert \det(m(\ou))\rvert^s \delta(m(\ou))^{1/2}\langle \pi(m(\ou))v,v^\vee\rangle d\ou$$
	for $(v,v^\vee)\in V_\pi\times V_\pi^\vee$ and $\Re(s)>0$. Let $\Xi^M$ be Harish-Chandra basic spherical function on the group $M$ \cite[\S II]{WaldPlanch}. Then, we have the estimates $\lvert \langle \pi(m)v,v^\vee\rangle\rvert\ll \Xi^M(m)$, for $m\in M$, \cite{CoHaHo} and the convergence can then be deduced from that of the integral
	\begin{equation}\label{eq20}
	\displaystyle \int_{\oN} \lvert \det(m(\ou))\rvert^s \delta(m(\ou))^{1/2} \Xi^M(m(\ou)) d\ou
	\end{equation}
	which follows e.g. from a combination of Lemmas II.3.4 and II.4.2 from \cite{WaldPlanch}.
	
	It follows from \eqref{eq5} and the first part of Proposition \ref{prop} that
	\begin{equation}\label{eq15}
	\displaystyle \int_{\oN}\Phi_{s}(\ou) d\ou=\int_{\oN} \int_{\Temp_\chi(M)} \Tr \Phi_{\pi,s}(\ou) d\mu_{M,\chi}(\pi) d\ou
	\end{equation}
for $\Re(s)>0$.

\begin{lem}\label{lem abs conv}
For every $s\in \bC$ with $\Re(s)>0$, the expression on the right-hand side of \eqref{eq15} is absolutely convergent.
\end{lem}

\begin{proof}
Let $J\subset M$ be a compact-open subgroup leaving $\Phi$ invariant on the left. Then, for every $g\in G$ and $\pi\in \Temp_\chi(M)$, the operator $\Phi_{\pi,s}(g)$ has image in $V_\pi^J$. Let us fix a basis $\mathcal{B}_\pi^J$ of the finite-dimensional vector space $V_\pi^J$ and let $(\mathcal{B}_\pi^{J})^\vee=\{ v^\vee\mid v\in \mathcal{B}_\pi^J\}$ be the dual basis of $(V_\pi^J)^*=(V_\pi^\vee)^J$. Then,
\[\begin{aligned}
\displaystyle \Tra \Phi_{\pi,s}(g) & =\sum_{v\in \mathcal{B}_\pi^J} \int_{M/A} \Phi_s(m^{-1}g) \delta(m)^{1/2} \lvert \det(m)\rvert^s \langle \pi(m)v,v^\vee\rangle dm \\
 & =\sum_{v\in \mathcal{B}_\pi^J} \int_{M} \Phi(m^{-1}g) \delta(m)^{1/2} \lvert \det(m)\rvert^s \langle \pi(m)v,v^\vee\rangle dm.
\end{aligned}\]
We fix as before a special maximal compact subgroup $K\subset G$ and we choose, for $\ou\in \oN$, $m(\ou)\in M$ and $k(\ou)\in K$ such that $\ou\in Nm(\ou)k(\ou)$. Since $\Phi$ is $N$-invariant on the left and by the change of variable $m\mapsto m(\ou)m$, we get
\begin{align}\label{eq19}
\displaystyle \Tra \Phi_{\pi,s}(\ou)& =\delta(m(\ou))^{1/2}\lvert \det(m(\ou))\rvert^s \sum_{v\in \mathcal{B}_\pi^J} \int_{M} \Phi(m^{-1}k(\ou)) \delta(m)^{1/2} \lvert \det(m)\rvert^s \langle \pi(m(\ou)m)v,v^\vee\rangle dm \\
\nonumber & =\delta(m(\ou))^{1/2}\lvert \det(m(\ou))\rvert^s \sum_{v\in \mathcal{B}_\pi^J} \langle \pi(m(\ou)) \pi(\varphi_{s,k(\ou)})v,v^\vee\rangle,
\end{align}
where, for $k\in K$, we have denoted by $\varphi_{s,k}\in C_c^\infty(M)$ the function $m\mapsto \Phi(m^{-1}k(\ou)) \delta(m)^{1/2} \lvert \det(m)\rvert^s$.

As $k$ runs over $K$, the function $\varphi_{s,k}$ describes a finite set. Thus by \cite{CoHaHo}, we have an estimate
$$\displaystyle \sum_{v\in \mathcal{B}_\pi^J} \lvert \langle \pi(m) \pi(\varphi_{s,k})v,v^\vee\rangle\rvert\ll \Xi^M(m),$$
for $m\in M$, $k\in K$ and $\pi\in \Temp(M)$; the implicit multiplicative constant being in particular independent of $\pi$. In particular, combining this with \eqref{eq19}, we get
$$\displaystyle \lvert \Tra \Phi_{\pi,s}(\ou)\rvert\ll \delta(m(\ou))^{1/2}\lvert \det(m(\ou))\rvert^{\Re(s)} \Xi^M(m(\ou)).$$

We know from Harish-Chandra \cite[th\'eor\`eme VIII.1.2]{WaldPlanch} that there are only finitely many connected components $\mathcal{O}\subset \Temp(M)$ such that $\pi^J\neq 0$ for some $\pi\in \mathcal{O}$. In particular, the function $\pi\mapsto \Tra \Phi_{\pi,s}(\ou)$ is supported on a compact subset of $\Temp_\chi(M)$, independent of $\ou$ (and $s$). From this, the lemma reduces again to the convergence of \eqref{eq20} for $\Re(s)>0$.
\end{proof}
	
From the previous lemma as well as the identification \eqref{eq6}, for $\Re(s)>0$ we have
	\begin{equation}\label{eq7}
		\displaystyle \int_{\oN} \Phi_s(\ou)d\ou=\int_{\Temp_\chi(M)}\int_{\oN} \Tr \Phi_{\pi,s}(\ou) d\ou d\mu_{M,\chi}(\pi)=\int_{\Temp_\chi(M)} \Tr((\cM(\pi,s)\Phi_{\pi,s})(1)) d\mu_{M,\chi}(\pi)
	\end{equation}
	where, in the last expression, we have used the slight abuse of notation of denoting by $\cM(\pi,s)\Phi_{\pi,s}$ the image of $ \Phi_{\pi,s}$ by the operator
	$$\displaystyle \cM(\pi,s)\otimes \Id: \rI_{P}^G(\pi_s)\otimes (\pi_s)^\vee\to \rI_{\overline{P}}^G(\pi_s)\otimes (\pi_s)^\vee\simeq \rI_{\overline{P}}^G(\pi_s\otimes (\pi_s)^\vee).$$
	
	Set 
	$$\displaystyle r(\pi,s)=\epsilon^{\Sh}(s,\pi,\Sym^2,\psi)\frac{L^{\Sh}(1-s,\pi^\vee,\Sym^2)}{L^{\Sh}(s,\pi,\Sym^2)},\; \pi\in \Temp(M), \; s\in \bC,$$
	where $\epsilon^{\Sh}(s,\pi,\Sym^2,\psi)$ and $L^{\Sh}(s,\pi,\Sym^2)$ stand for the local $\epsilon$- and $L$-factors of the symmetric square defined by Shahidi \cite{Sha90}. We know from \cite{HenLfn}, that they coincide with their Galois counterparts $L(s,\pi,\Sym^2)$ and $\epsilon(s,\pi,\Sym^2,\psi)$ (defined through the local Langlands correspondence for $M$), in the latter case up to a unit, i.e. there exists $u_1(\pi)\in \mathbb{S}^1$ such that
	\begin{equation}\label{eq17}
		\displaystyle r(\pi,s)=u_1(\pi)\cdot \gamma(s,\pi,\Sym^2,\psi).
	\end{equation}
	Moreover, by \cite[Theorem 7.7]{Sha90}, the normalised intertwining operators
	$$\displaystyle \cN(s,\pi):=r(s,\pi)\cM(s,\pi)$$
	are regular and unitary (up to an absolute constant depending on the choice of Haar measure on $N$) for $s$ imaginary. In particular, \eqref{eq7} can be rewritten as
	\begin{equation*}
		\displaystyle \int_{\oN} \Phi_s(\ou)d\ou=\int_{\Temp_\chi(M)} \Tr((\cN(\pi,s)\Phi_{\pi,s})(1)) r(s,\pi)^{-1}d\mu_{M,\chi}(\pi)
	\end{equation*}
	where, for $s$ sufficiently close to the imaginary line $i\bR$, the function $\pi\mapsto \Tr(\cN(\pi,s)f_{\pi,s}(1))$ is an element of $C_c^\infty(\Temp_\chi(M))$ that varies continuously with $s$ (this can be checked using the criterion at the end of Remark \ref{rmk}). Thus, from Proposition \ref{prop1 spectral limit} as well as Remark \ref{rmk}, we obtain
	\begin{equation}\label{eq18}
		\displaystyle \lim\limits_{s\to 0} \gamma(s,\mathbf{1}_F,\psi) \int_{\oN} \Phi_s(\ou)d\ou=\frac{2}{d}\chi(-1)^{d-1} \int_{\Temp_\chi^{\orth}(M)} \Tr((\cN(\pi,0)\Phi_{\pi,0})(1)) u_1(\pi)\frac{\gamma^*(0,\pi,\wedge^2,\psi)}{\lvert S_\pi^{+}\rvert} d\pi.
	\end{equation}
	
	Fix $w\in \tM$ and let $\pi\in \Temp_\chi^{\orth}(M)$. Then, we have the intertwining operators
	$$\displaystyle \rI_{P}^G(\tpi(w)):\rI_{P}^G(\pi)\to \rI_{P}^G(\pi^w),\;\; L(w): \rI_{P}^G(\pi^w)\to \rI_{\overline{P}}^G(\pi)$$
	given by $(\rI_{P}^G(\tpi(w))e)(g)=\tpi(w)e(g)$ and $(L(w)e)(g)=e(w^{-1}g)$ respectively. Note that both the operators $\rI_{P}^G(\tpi(w))$ and $\tilde{\delta}(w)L(w)$ are unitary. For the first operator because we have chosen $\tpi$ to be unitary and  in the second case by definition of the scalar products on $\rI_{P}^G(\pi^w)$, $\rI_{\overline{P}}^G(\pi)$ as well as the definition of $\tilde{\delta}(w)$. Moreover, by Proposition \ref{prop irred generic}, for $\pi\in \Temp_\chi^{\orth}(M)$ in general position, the induced representations $\rI_P^G(\pi)$, $\rI_{\overline{P}}^G(\pi)$ are irreducible and it follows, by Schur lemma, that the unitary intertwining operator $\cN(\pi,0)$ coincides up to a unit with the composition $\tilde{\delta}(w)L(w)\circ \rI_{P}^G(\tpi(w))$ i.e. we can find $u_2(\tpi)\in \mathbb{S}^1$ such that
	$$\displaystyle \cN(\pi,0)=u_2(\tpi)\tilde{\delta}(w) L(w)\circ \rI_{P}^G(\tpi(w)).$$
	This implies that, for every $v\in V_\pi$, we have
	\[\begin{aligned}
		\displaystyle (u_2(\tpi)\tilde{\delta}(w))^{-1}(\cN(\pi,0)\Phi_{\pi,0})(1)v & = \tpi(w) \Phi_{\pi,0}(w^{-1})v= \tpi(w)  \int_{M/A} \Phi_0(m^{-1}w^{-1}) \delta(m)^{1/2} \pi(m)v dm \\
		& =\int_{M} \Phi(m^{-1}w^{-1}) \delta(m)^{1/2} \tpi(wm)v dm \\
		& = \widetilde{\delta}(w)^{-1}\int_{\tM} \Phi^\vee(\tm) \widetilde{\delta}(\tm)^{1/2} \tpi(\tm)v dm= \widetilde{\delta}(w)^{-1} \tpi(f)v,
	\end{aligned}\]
	where the last line follows from \eqref{eq tf}. This allows to rewrite \eqref{eq18} as
	\[\begin{aligned}
		\displaystyle \lim\limits_{s\to 0} \gamma(s,\mathbf{1}_F,\psi) \int_{\oN} \Phi_s(\ou)d\ou =C_2\int_{\Temp_\chi^{\orth}(M)} \Theta_{\tpi}(f) u_3(\tpi)\frac{\gamma^*(0,\pi,\wedge^2,\psi)}{\lvert S_\pi^{\orth}\rvert} d\pi
	\end{aligned}\]
	where $C_2:=\frac{2}{d}$ and $u_3(\tpi):=u_1(\pi)u_2(\tpi)\chi(-1)^{d-1}$.
	
	Combining the above identity with Proposition \ref{prop} proves Theorem \ref{theo spectraldec orbint}.

	\section{Twisted endoscopy and Plancherel densities}\label{Sect twisted endoscopy}
	
	\subsection{Symplectic and even orthogonal groups}
	
	Let $n\geq 1$ be an integer and let $W$ be vector space over $F$ of dimension $2n$ that is equipped with a nondegenerate symmetric or alternating form $q$. We let
	$$\displaystyle H=G(W,q)^0$$
	be the neutral component of the group of linear automorphisms of $W$ preserving $q$. Thus if $q$ is symmetric, $H$ is an even special orthogonal group whereas if $q$ is alternating it is a symplectic group.
	
	Assume first that $q$ is symmetric. In this case, we denote by $\disc(q)$ its discriminant normalised in the following way: if $(w_i)_{i=1}^{2n}$ is a basis of $W$ then $\disc(q)$ is the following square class 
	$$\displaystyle \disc(q):=(-1)^n \det((q(w_i,w_j))_{i,j})\in F^\times/F^{\times,2}.$$
	We set $E=F[\sqrt{\disc(q)}]$. Thus, $E=F$ if $\disc(q)=1$ and otherwise $E/F$ is a quadratic extension. We denote by $\chi: F^\times \to \{ \pm 1\}$ the quadratic character associated to this extension, that is the unique character with kernel the image of the norm map $N_{E/F}:E^\times\to F^\times$. As it will be convenient at few points, we also fix an element $g^+\in G(W,q)\setminus G(W,q)^0$ of order $2$, where $G(W,q)=\mathrm{O}(W,q)$ denotes the full orthogonal group of the quadratic space $(W,q)$, and by $\theta: g\mapsto g^+ gg^+$ the corresponding involutive automorphism of $H$. Of course, the class of $\theta$ modulo the group of inner automorphisms of $H$ does not depend on the choice of $g^+$. 
	
	Still assuming that $q$ is symmetric, we also fix another quadratic space $(W^*,q^*)$ of dimension $2n$ such that $\disc(q^*)=\disc(q)$ and $H^*:=G(W^*,q^*)^0$ is quasi-split. Then, $H^*$ is a pure inner form of $H$. We fix similarly an involutive automorphism $\theta^*$ of $H^*$ corresponding to the choice of some element $(g^*)^+\in G(W^*,q^*)\setminus G(W^*,q^*)^0$ of order $2$.
	
	Assume now that $q$ is symplectic. Then, all of these constructions are not necessary but in order to allow for a more uniform treatment we will simply set $E=F$, $\chi=1$, $H=H^*$ and $\theta=\theta^*=1$.
	
	Set
	$$\displaystyle d=\left\{\begin{array}{ll}
		2n & \mbox{ if } q \mbox{ is symmetric,} \\
		2n+1 & \mbox{ if } q \mbox{ is alternating.}
	\end{array} \right.$$
	Let $V$ be a vector space over $F$ of dimension $d$ and denote by $(M,\tM)$ the corresponding twisted space defined as in Subsection \ref{sect twisted linear group}. We again denote by $A$ the center of $M$, identified in the usual way with $F^\times$; in particular we will identify the quadratic character $\chi$ with a character of $A$.
	
	\subsection{Endoscopic transfers}
	
	The quasi-split group $H^*$ is both an endoscopic group of $H$ and of the twisted space $\tM$. More precisely, in the case of $H$ this just come from the equality of $L$-groups ${}^L H={}^L H^*$ (and the fact that $H^*$ is quasi-split) whereas for $\tM$ we can fix an endoscopic datum as in \cite{Waldtransfact}. This, in particular, entails a notion of {\it endoscopic transfer} from test functions on $\tM$ (or $H$) to test functions on $H^*$ that we recall below.
	
	\vspace{2mm}
	
	\begin{paragr}[Geometric correspondences.]
		
		We have correspondences between sufficiently regular semisimple conjugacy classes in, on the one hand, $H$ and $\tM$ and, on the other hand, $H^*$.
		
		Recall that an element $\delta\in H$ is said to be {\em semisimple strongly regular} if its centralizer $H_\delta=Z_H(\delta)$ is a torus. Semisimple strongly regular elements of $H^*$ are defined in the same way. We denote by $H_{\sr}\subset H$ and $H^*_{\sr}\subset H^*$ the open subsets of semisimple strongly regular elements.
		
		The notion of semisimple strongly regular element in the twisted space $\tM$ is defined as follows: we say that $\gamma\in \tM$ is {\em semisimple strongly regular} if its centralizer $Z_M(\gamma)$ in $M$ is abelian and its neutral component $M_\gamma=Z_M(\gamma)^0$ is a torus. We again denote by $\tM_{\sr}\subset \tM$ the open subset of semisimple strongly regular elements.
		
		First consider the correspondence between semisimple strongly regular conjugacy classes in $H$ and $H^*$. There is a natural inner class of isomorphisms $H_{\overline{F}}\simeq H^*_{\overline{F}}$, namely those induced from isomorphisms of symplectic or quadratic spaces $(W,q)\otimes_F \overline{F}\simeq (W^*,q^*)\otimes_F \overline{F}$, and two strongly regular semisimple elements $\delta^*\in H^*_{\sr}$ and $\delta\in H_{\sr}$ {\em correspond to each other} if their $H^*_{\overline{F}}$- and $H_{\overline{F}}$-conjugacy classes are mapped to each other by such isomorphism.

		On the other hand, the correspondence between (semisimple strongly regular) conjugacy classes in $\tM$ and $H^*$ can be explicitely described using characteristic polynomials as follows. Let $F[T]_{u,d}$ be the space of unitary polynomials of degree $d$ over $F$. We have two maps
		\begin{equation}\label{GIT correspondence}
			\displaystyle \cC_{H^*}:\mathrm{H}^* \to F[T]_{u,d} \mbox{ and }  \cC_{\tM}:\widetilde{\mathrm{M}}\to F[T]_{u,d}
		\end{equation}
	defined as follows: for $\gamma\in \tM$, that we view as an isomorphism $\gamma: V\to V^*$,
	$$\cC_{\tM}(\gamma)=\det(T\Id_V-{}^t \gamma^{-1}\gamma)$$
	is the characteristic polynomial of ${}^t \gamma^{-1}\gamma$ (seen as a linear endomorphism of $V$), whereas, for $\delta^*\in H^*$, denoting by
	$$\chi_{\delta^*}=\det(T\Id_W-\delta^*)$$
	the characteristic polynomials of $\delta^*$ (seen as a linear endomorphism of $W^*$),
	\begin{equation}\label{corresp polchar}
		\displaystyle \cC_{H^*}(\delta^*)=\left\{\begin{array}{ll}
			\chi_{\delta^*}(-T) & \mbox{ if }  q \mbox{ is symmetric,} \\
			\chi_{\delta^*}(T)(T-1) & \mbox{ if } q \mbox{ is alternating.}
		\end{array} \right.
	\end{equation}
	Then, two strongly regular semisimple elements $\delta^*\in H^*_{\sr}$ $\gamma\in \tM_{\sr}$ are said to {\em correspond to each other} if $\cC_{H^*}(\delta^*)=\cC_{\tM}(\gamma)$.
		
	\end{paragr}

	\begin{paragr}[Transfer factors.]
		We can define, following \cite{KottShel}, {\em transfer factors} between $H^*$ and $\tM$, that is a function
		$$\displaystyle \Delta: H^*_{\sr}\times \tM_{\sr}\to \mathbb{C}$$
		with the property that $\Delta(\delta^*,\gamma)=0$ unless $\delta^*\in H^*_{\sr}$ and $\gamma\in \tM_{\sr}$ correspond to each other. Since our calculations will all be done up to an absolute constant, we won't need to deal with the delicate issue of normalizing precisely these transfer factors. More precisely, we simply require that for every $(\delta^*_1,\delta^*_2,\gamma_1,\gamma_2)\in H^*_{\sr}\times H^*_{\sr} \times \tM_{\sr}\times \tM_{\sr}$ with $\delta^*_1$ corresponding to $\gamma_1$ and $\delta^*_2$ corresponding to $\gamma_2$, the ratio
		$$\displaystyle \Delta(\delta^*_1,\gamma_1)\Delta(\delta^*_2,\gamma_2)^{-1}$$
		is equal to the canonical transfer bifactor defined in  \cite[Chap. I, Sect. 2.2]{MWStab1}. In particular, unlike Kottwitz-Shelstad \cite{KottShel}, we do not include in this bifactor the usual ratio of Weyl discriminants (denoted by $\Delta_{IV}$ in {\em loc.\ cit.}\ ).
		
		We have the following simple formula for the behavior of transfer factors under multiplication by the center $A$ of $M$ (see \cite[lemme I.2.7]{MWStab1} for a much more general statement): for every $(\delta^*,\gamma)\in H^*_{\sr}\times \tM_{\sr}$ and $a\in A$ we have
		\begin{equation}\label{central character transfer factors}
			\Delta(\delta^*,a\gamma)=\chi(a)\Delta(\delta^*,\gamma).
		\end{equation}
		
		\begin{rem}
			It turns out that when $H$ is a symplectic, the transfer factor $\Delta$ is constant on pairs of matching regular semisimple elements and can therefore be taken to be identically $1$ there. This stems from the fact that in this case $H$ is the {\em principal endoscopic group} of $\tM$. However, we will not use this fact explicitely in the sequel.
			Similarly, there are transfer factors between $H^*$ and $H$ but these are constant equal to one and so we don't need to introduce them formally.
		\end{rem}
	\end{paragr} 
	
	\begin{paragr}[Smooth transfer.]
		First, we recall the notion of stable orbital integrals for functions on $H^*$. Let $f^{H^*}\in C_c^\infty(H^*)$ and $\delta^*\in H^*_{\sr}$. Then, the {\em stable orbital integral} of $f^{H^*}$ at $\delta^*$ is defined as
		$$\displaystyle \SI_{\delta^*}(f^{H^*})=\sum_{\epsilon^*\sim_{\stab} \delta^*} \rI_{\epsilon^*}(f^{H^*}),$$
		where the sum runs over $H^*$-conjugacy classes of elements $\epsilon^*\in H^*$ in the same {\em stable conjugacy class} as $\delta^*$, that is elements which are conjugated to $\delta^*$ in $H^*_{\overline{F}}$, and $\rI_{\epsilon^*}(f^{H^*})$ denotes the orbital integral of $f^{H^*}$ on the conjugacy class of $\epsilon^*$ normalised by
		$$\displaystyle \rI_{\epsilon^*}(f^{H^*})=D^{H^*}(\epsilon^*)^{1/2}\int_{H^*_{\epsilon^*}\backslash H^*} f^{H^*}({h}^{-1}\epsilon^* h) \frac{dh}{dh_{\epsilon^*}},$$
		i.e. including the square root of the Weyl discriminant
		$$\displaystyle D^{H^*}(\epsilon^*)=\left\lvert \det(1-\Ad(\epsilon^*)\mid \mathfrak{h}^*/\mathfrak{h}^*_{\epsilon^*})\right\rvert.$$
		These orbital integrals depend implicitely on the choice of Haar measures on the centralizers $H_{\epsilon^*}^*$ (as well as on $H^*$) and we take the $\psi$-Haar measures as recalled in \S \ref{S measures}. The important property here being that these measures are {\em compatible} i.e. that they correspond to one chosen Haar measure on $H^*_{\delta^*}$ by the isomorphisms $H^*_{\epsilon^*}\simeq H^*_{\delta^*}$ induced by the conjugation action of any element $h\in H^*_{\overline{F}}$ sending $\epsilon^*$ to $\delta^*$. (This follows directly from the definition of the $\psi$-Haar measures since these isomorphisms come from isomorphisms of linear groups defined over $F$.)
		
		We define in a similar way the stable orbital integrals $\SI_{\delta}(f^H)$ for $f^H\in C_c^\infty(H)$ and $\delta\in H_{\sr}$.
		
		Then, a function $f^{H^*}\in C_c^\infty(H^*)$ is said to be a {\em transfer} of a function $f^H\in C_c^\infty(H)$ if for every matching elements $(\delta^*,\delta)\in H^*_{\sr}\times H_{\sr}$, we have the equality of stable orbital integrals
		\begin{equation*}
			\displaystyle \SI_{\delta^*}(f^{H^*})=\SI_{\delta}(f^H),
		\end{equation*}
		whereas if $\delta^*\in H^*_{\sr}$ does not correspond to any element $\delta\in H$,
		\begin{equation*}
			\displaystyle \SI_{\delta^*}(f^{H^*})=0.
		\end{equation*}
		
		In a similar vein, a function $f^{H^*}\in C_c^\infty(H^*)$ is said to be a {\em transfer} of a function $f^{\tM}\in C_c^\infty(\tM)$ if for every $\delta^*\in H^*_{\sr}$ we have
		$$\displaystyle \SI_{\delta^*}(f^{H^*})=\sum_{\gamma} \Delta(\delta^*,\gamma) \rI_\gamma(f^{\tM}),$$
		where the sum runs over $M$-conjugacy classes in $\tM_{\sr}$ corresponding to $\delta^*$ and the orbital integrals $\rI_\gamma(f^{\tM})$ are given by
		$$\displaystyle \rI_\gamma(f^{\tM})=D^{\tM}(\gamma)^{1/2}\int_{M_\gamma\backslash M} f^{\tM}(m^{-1}\gamma m) \frac{dm}{dm_\gamma},$$
		with
		$$\displaystyle D^{\tM}(\gamma)=\left\lvert \det(1-\Ad(\gamma)\mid \mathfrak{m}/\mathfrak{m}_\gamma)\right\rvert.$$
		Once again, the Haar measures in the above integral are chosen to be the $\psi$-Haar measures.
		
		From \eqref{central character transfer factors}, we deduce the following property: if $f^{H^*}$ is a transfer of $f^{\tM}$ then it is also a transfer of $\chi(a)L(a)f^{\tM}$ for every $a\in A$.
		
		Thanks to the proof by Ng\^o of the fundamental lemma \cite{Ngo} and the work of Waldspurger \cite{Walchgtcar} \cite{Waldendtordue}, the existence of endoscopic transfer is known to exist in full generality. We summarize this as well as the description of the image of transfer from $\tM$ (see \cite[Corollary 2.1.2]{artbook} or \cite[proposition I.4.11]{MWStab1} for a general statement) in the following theorem.
		
		\begin{theo}[Ng\^o, Waldspurger]\label{prop image transfer}
			Every function $f^H\in C_c^\infty(H)$ (resp. $f^{\tM}\in C_c^\infty(\tM)$) admits a transfer $f^{H^*}\in C_c^\infty(H^*)$. Moreover, the image of the transfer $f^{\tM}\to f^{H^*}$ consists of all functions with stable orbital integrals invariant by $\theta^*$, in other words: a function $f^{H^*}\in C_c^\infty(H^*)$ is the transfer of some function $f^{\tM}\in C_c^\infty(\tM)$ if and only if for every $\delta^*\in H^*$ we have
			$$\SI_{\theta^*(\delta^*)}(f^{H^*})=\SI_{\delta^*}(f^{H^*}).$$
			In particular, the subspace of $\theta^*$-invariant functions $C_c^\infty(H^*)^{\theta^*}\subset C_c^\infty(H^*)$ is in the image of the transfer.
		\end{theo}
		
	\end{paragr}

	\subsection{Transfer of central Diracs}\label{Sect transfer dirac}
	
	Let $\epsilon\in Z(H)=Z(H^*)=\{ 1,-1\}$ be the central element given by
	$$\displaystyle \epsilon=(-1)^d.$$
	Thus, $\epsilon=+1$ in the special orthogonal case whereas $\epsilon=-1$ in the symplectic case.
	
	\begin{prop}\label{prop transfer dirac}
		\begin{enumerate}[(i)]
			\item For every function $f^H\in C_c^\infty(H)$ with transfer $f^{H^*}\in C_c^\infty(H^*)$ we have
			\begin{equation*}
				\displaystyle f^{H^*}(\epsilon)=f^H(\epsilon).
			\end{equation*}
		
		\item There exists a constant $c\in \mathbb{C}^\times$ such that for every $f^{\tilde{M}}\in C_c^\infty(\tilde{M})$ with transfer $f^{H^*}\in C_c^\infty(H^*)$,
		\begin{equation*}
			\displaystyle f^{H^*}(\epsilon)=c\rI_{\chi}(f^{\tM}),
		\end{equation*}
		where we recall that the distribution on the right hand side was defined by \eqref{def Ichi}.
		
	\end{enumerate}
	\end{prop}
	
	\begin{proof}
		$(i)$ follows from \cite[Proposition 2]{KottTam} noting that the Kottwitz signs $e(H)$, $e(H^*)$ appearing in {\em loc.\ cit.}\ are both equal to $1$.
		Let us prove $(ii)$. By \cite[Proposition 1]{KottTam}, the distribution $f^{H^*}\mapsto f^{H^*}(\epsilon)$ is stable, which means that, for $f^{H^*}\in C_c^\infty(H^*)$, if $\SI_{\delta^*}(f^{H^*})=0$ for every strongly regular semisimple element $\delta^*\in H^*$ then $f^{H^*}(\epsilon)=0$. From that, and the definition of endoscopic transfer, we deduce the existence of an invariant distribution $\rJ$ on $\tM$ such that $f^{H^*}(\epsilon)=\rJ(f^{\tM})$ for every $f^{\tilde{M}}\in C_c^\infty(\tilde{M},\chi)$ with transfer $f^{H^*}\in C_c^\infty(H^*)$. Moreover, by Theorem \ref{prop image transfer}, this distribution is nonzero since there certainly exists a $\theta^*$-invariant function $f^{H^*}\in C_c^\infty(H^*)$ such that $f^{H^*}(\epsilon)\neq 0$. It remains to show that $\rJ$ is proportional to $\rI_{\chi}$.
		
		Since the orbit correspondence is described by the morphisms \eqref{GIT correspondence}, for every compact-open subset $\mathcal{U}\subset F[T]_{u,d}$, and every pair of matching functions $(f^{H^*},f^{\tM})\in C_c^\infty(H^*)\times C_c^\infty(\tM)$, the functions $\mathbf{1}_{\mathcal{C}_{H^*}^{-1}(\mathcal{U})} f^{H^*}$ and $\mathbf{1}_{\mathcal{C}_{\tM}^{-1}(\mathcal{U})} f^{\tM}$ also match. (Here, $\mathbf{1}_{\mathcal{C}_{H^*}^{-1}(\mathcal{U})}$, $\mathbf{1}_{\mathcal{C}_{\tM}^{-1}(\mathcal{U})}$ stand for the characteristic of the open subsets $\mathcal{C}_{H^*}^{-1}(\mathcal{U})\subseteq H^*$ and $\mathcal{C}_{\tM}^{-1}(\mathcal{U})\subseteq \tM$ respectively.) This implies that $\rJ$ is supported on the inverse image of $\mathcal{C}_{H^*}(\epsilon)$ by $\mathcal{C}_{\tM}$.
		
		By the explicit description of $\mathcal{C}_{H^*}$ and $\mathcal{C}_{\tM}$ in terms of characteristic polynomials, we see that this fiber is exactly the subset of those elements $\gamma\in \tM$ such that the characteristic polynomial $\chi_{{}^t \gamma^{-1}\gamma}$ of ${}^t \gamma^{-1}\gamma$ is given by
		\begin{equation}\label{eq polchar2}
			\displaystyle \chi_{{}^t \gamma^{-1}\gamma}=\left\{\begin{array}{ll}
				(T+1)^{2n} & \mbox{ if } q \mbox{ is quadratic,} \\
				(T+1)^{2n}(T-1) & \mbox{ if } q \mbox{ is alternating.}
			\end{array}
			\right.
		\end{equation}
		
		We now distinguish between the orthogonal and symplectic case. First assume that $H$ is a special orthogonal group (that is $q$ is symmetric). Choose $\gamma_0\in \tM$ that corresponds to a nondegenerate alternating form on $V$ and let $M_{\gamma_0}$ be its centralizer in $M$ (that is the corresponding symplectic group). Then, the $\gamma\in \tM$ satisfying \eqref{eq polchar2} are exactly those whose semisimple part is $M$-conjugate to $\gamma_0$, in other words these are exactly the $M$-conjugates of elements of the form $\gamma_0 u$ where $u\in M_{\gamma_0}$ is unipotent. More precisely, the map $u\mapsto \gamma_0 u$ set up a bijection between unipotent conjugacy classes in $M_{\gamma_0}$ and $M$-conjugacy classes of elements $\gamma\in \tM$ with $\chi_{{}^t \gamma^{-1}\gamma}=(T+1)^{2n}$. In particular, if we denote by $\mathrm{Unip}(M_{\gamma_0})$ the (finite) set of unipotent conjugacy classes in $M_{\gamma_0}$, the distribution $\rJ$ can be written in a unique way as a linear combination
		\begin{equation}\label{exp unip orb}
			\displaystyle \rJ(f^{\tM})=\sum_{u\in \mathrm{Unip}(M_{\gamma_0})} c_u I_{\gamma_0u}(f^{\tM}),\;\; f^{\tM}\in C_c^\infty(\tM),
		\end{equation}
		where
		$$\displaystyle I_{\gamma_0u}(f^{\tM})=\int_{M_{\gamma_0 u}\backslash M} f(m^{-1}\gamma_0 um) dm$$
		denotes the orbital integral on the conjugacy class of $\gamma_0 u$.
		
		Let $\omega\subset \mathfrak{m}_{\gamma_0}(F)$ and $\omega'\subset \mathfrak{h}^*(F)$ (where $\mathfrak{m}_{\gamma_0}$, $\mathfrak{h}^*$ denote the Lie algebras of $M_{\gamma_0}$ and $H^*$ respectively) be $M_{\gamma_0}$-invariant and $H^*$-invariant open-closed neighborhood of $0$ that are small enough that the exponential maps are well-defined on them and induce $M_{\gamma_0}$-equivariant and $H^*$-equivariant open embeddings
		$$\exp: \omega\to M_{\gamma_0},\;\;\; \exp: \omega'\to H^*$$
		respectively. We assume moreover, as we may, that 
		\begin{itemize}
			\item $\omega$ and $\omega'$ are invariant by multiplication by $\mathcal{O}_F$;
			
			\item for $X\in \omega'$ (resp. $Y\in \omega$), $\exp(X)\in H^*_{\sr}$ (resp. $\gamma_0\exp(Y)\in \tM_{\sr}$) if and only if $X$ is semisimple regular (resp. $Y$ is semisimple regular);
			
			\item Denoting by $\omega_{\sr}\subset \omega$, $\omega'_{\sr}\subset \omega'$ the open subsets of semisimple regular elements, for every $X\in \omega'$, $Y\in \omega$, we have the identity of Weyl discriminants
			\begin{equation}\label{eq Weyl disc}
				\displaystyle D^{H^*}(\exp(X))=D^{H^*}(X) \mbox{ and } D^{\tM}(\gamma_0 \exp(Y))=D^{\tM}(\gamma_0) D^{M_{\gamma_0}}(Y),
			\end{equation}
			where
		$$\displaystyle D^{H^*}(X)=\left\lvert \det(1-\ad(X)\mid \mathfrak{h}^*/\mathfrak{h}^*_X)\right\rvert,$$
		$$\displaystyle D^{M_{\gamma_0}}(Y)=\left\lvert \det(1-\ad(Y)\mid \mathfrak{m}_{\gamma_0}/\mathfrak{m}_{\gamma_0,Y})\right\rvert.$$
		\end{itemize}
		
		The group $H^*$ can be naturally extended to an endoscopic datum for $M_{\gamma_0}$ (such datum can actually be deduced from the endoscopic datum presenting $H^*$ as an endoscopic group of $\tM$). Moreover, the theory of endoscopy extends to Lie algebras \cite[\S 2]{WaldLFimptrans}. In the present case, this means that there is a natural correspondence between (semisimple regular) adjoint orbits in $\mathfrak{m}_{\gamma_0}(F)$ and $\mathfrak{h}^*(F)$ as well as a Lie algebra analog of transfer factors
		$$\displaystyle \Delta^{\Lie}: \mathfrak{h}^*_{\rs}(F)\times \mathfrak{m}_{\gamma_0,\rs}(F)\to \mathbb{C},$$
		which is nonzero only on pairs of corresponding elements. As for the previous transfer factors, we remove from the original definition the usual ratio of Weyl discriminants $\Delta_{IV}$, which implies that $\Delta^{\Lie}$ satisfies the following homogeneity property: for every $\lambda\in F^\times$ and $(X,Y)\in \mathfrak{h}^*(F)\times \mathfrak{m}_{\gamma_0}(F)$,
		\begin{equation}\label{homog transf factors}
			\Delta^{\Lie}(\lambda^2 X,\lambda^2 Y)=\Delta^{\Lie}(X,Y).
		\end{equation}
There is also a compatibility with transfer factors on the groups. Namely, we can choose $\omega$ and $\omega'$ such that the following conditions are satisfied \cite[lemme 3.8, théorème 3.9]{Waldendtordue}:
		\begin{itemize}
			\item for every $X\in \omega'_{\sr}$, the elements of $\tM$ corresponding to $\exp(X)\in H^*$ are exactly the $M$-conjugate of elements of the form $\gamma_0\exp(Y)$ where $Y\in \omega_{\sr}$ corresponds to $X$ and conversely, for every $Y\in \omega_{\sr}$ the elements of $H^*$ corresponding to $\gamma_0 \exp(Y)$ are exactly the conjugacy classes of elements the form $\exp(X)$ where $X\in \omega'_{\sr}$ corresponds to $Y$;
			
			\item for every $(X,Y)\in \omega_{\sr}\times \omega'_{\sr}$ corresponding to each other
			$$\Delta(\exp(X),\gamma_0\exp(Y))=\Delta^{\Lie}(X,Y).$$
		\end{itemize}
		
		Set
		$$\displaystyle \Omega':=\exp(\omega'),$$
		$$\displaystyle \Omega:=(\gamma_0 \exp(\omega))^M:=\left\{m\gamma_0 \exp(X)m^{-1},\; m\in M, X\in \omega\right\}.$$
		These are open-closed subsets of $H^*$ and $\tM$ respectively. Moreover, every function $f^{\tM}\in C_c^\infty(\Omega)$ admits a transfer $f^{H^*}\in C_c^\infty(\Omega')$ characterized by the identities
		\begin{equation}\label{transfer id lie alg}
			\displaystyle \SI_{\exp(X)}(f^{H^*})=\sum_{Y} \Delta^{\Lie}(X,Y) \rI_{\gamma_0 \exp(Y)}(f^{\tM}), \mbox{ for every } X\in \omega'_{\sr},
		\end{equation}
		where the sum runs over $M_{\gamma_0}$-conjugacy classes of elements $Y\in \omega_{\sr}$ corresponding to $X$.
		
		For every $f^{\tM}\in C_c^\infty(\Omega)$ and $t\in \mathcal{O}_F\setminus \{ 0\}$, define a new function $f^{\tM}_t\in C_c^\infty(\Omega)$ by
		$$\displaystyle f^{\tM}_t(m^{-1}\gamma_0 \exp(Y)m)=\left\{\begin{array}{ll}
			f^{\tM}(m^{-1}\gamma_0 \exp(t^{-1}X)m) & \mbox{ if } t^{-1}X\in \omega, \\
			0 & \mbox{ if } t^{-1}X\notin \omega,
		\end{array}
		\right.$$
		for every $(m,Y)\in M\times \omega$. Then, it follows from \eqref{eq Weyl disc} that for every $Y\in \omega_{\rs}$ we have
		\begin{equation}\label{eq1 orb int lie alg}
			\displaystyle \rI_{\gamma_0\exp(Y)}(f^{\tM}_t)=\left\{\begin{array}{ll}
				\lvert t\rvert^{\delta(M_{\gamma_0})/2} \rI_{\gamma_0\exp(t^{-1}Y)}(f^{\tM}) & \mbox{ if } t^{-1}Y\in \omega, \\
				0 & \mbox{ if } t^{-1}Y\notin\omega,
			\end{array} \right.
		\end{equation}
		where $\delta(M_{\gamma_0})=\dim(M_{\gamma_0})-\dim(T)$ for any maximal torus $T\subset M_{\gamma_0}$, that is the exponent such that $D^{M_{\gamma_0}}(tY)=\lvert t\rvert^{\delta(M_{\gamma_0})} D^{M_{\gamma_0}}(Y)$ for every $(t,Y)\in F^\times \times \mathfrak{m}_{\gamma_0,\rs}(F)$. 
		
		We define similarly a function $f^{H^*}_t\in C_c^\infty(\Omega')$ for every $f^{H^*}\in C_c^\infty(\Omega')$ and $t\in \mathcal{O}_F\setminus \{ 0\}$. It satisfies
		\begin{equation}\label{eq2 orb int lie alg}
			\displaystyle \SI_{\exp(X)}(f^{H^*}_t)=\left\{\begin{array}{ll}
				\lvert t\rvert^{\delta(H^*)/2} \SI_{\exp(t^{-1}X)}(f^{H^*}) & \mbox{ if } t^{-1}X\in \omega', \\
				0 & \mbox{ if } t^{-1}X\notin \omega',
			\end{array}
			\right.
		\end{equation}
		for every $X\in \omega'_{\rs}$, where this time $\delta(H^*)=\dim(H^*)-\dim(T^*)$ for any maximal torus $T^*\subset H^*$.
		
		Let $f^{\tM}\in C_c^\infty(\Omega)$ and $f^{H^*}\in C_c^\infty(\Omega')$ be a transfer of it. It follows from \eqref{homog transf factors} , \eqref{eq1 orb int lie alg}, \eqref{eq2 orb int lie alg} as well as the identities \eqref{transfer id lie alg} characterizing the transfer, that for every $t\in \mathcal{O}_F\setminus \{ 0\}$, $\lvert t\rvert^{\delta(M_{\gamma_0})-\delta(H^*)}f^{H^*}_{t^2}$ is a transfer of $f^{\tM}_{t^2}$. In particular, we have
		$$\displaystyle \rJ(f^{\tM}_{t^2})=\lvert t\rvert^{\delta(M_{\gamma_0})-\delta(H^*)}f^{H^*}_{t^2}(1)=\lvert t\rvert^{\delta(M_{\gamma_0})-\delta(H^*)}f^{H^*}(1)=\lvert t\rvert^{\delta(M_{\gamma_0})-\delta(H^*)}\rJ(f^{\tM}).$$
		We readily compute that
		$$\displaystyle \delta(M_{\gamma_0})-\delta(H^*)=\delta(\Sp_{2n})-\delta(\SO_{2n})=2n.$$
		On the other hand, the orbital integrals $\rI_{\gamma_0u}$, $u\in \mathrm{Unip}(M_{\gamma_0})$ have the following homogeneity property \cite[Lemma 3.2]{HCadminv}:
		$$\displaystyle \rI_{\gamma_0 u}(f^{\tM}_{t^2})=\lvert t\rvert^{\delta(u)}I_{\gamma_0 u}(f^{\tM}),\;\ t\in \mathcal{O}_F\setminus \{ 0\},$$
		where $\delta(u)$ denotes the dimension of the $M_{\gamma_0}$-conjugacy class of $u$. Comparing the two previous equations with \eqref{exp unip orb}, we see that for every $u\in \mathrm{Unip}(M_{\gamma_0})$, $c_u=0$ unless $\delta(u)=2n$. However, $2n$ is exactly the dimension of the minimal unipotent orbit of $M_{\gamma_0}=\Sp_{2n}$, which is the only geometric unipotent orbit of that dimension. Thus, only the orbital integrals at $\gamma_0 u$, for $u$ in the minimal unipotent orbit, can contribute to $\rJ$. It is immediate to check that the $M$-conjugacy class so obtained are exactly the $M$-conjugacy classes of the elements $\gamma_t$ for $t\in F^\times/F^{\times,2}$. It follows that
		\begin{equation}\label{roughdecomp I}
			\displaystyle \rJ(f^{\tM})=\sum_{t\in F^\times/F^{\times,2}} c_t I_{\gamma_t}(f^{\tM}),\;\; f^{\tM}\in C_c^\infty(\tM),
		\end{equation}
		for certain constants $c_t\in \bC$. Using now that $f^{\tM}$ and $\chi(a)L(a)f^{\tM}$ have the same transfers for every $a\in A=F^\times$, we see that $c_{ta}=\chi(a)c_t$ for every $a,t\in F^\times$ hence that $\rJ$ is indeed proportional to $\rI_{\chi}$.

		The case where $H$ is symplectic can be dealt with in a similar way. The main difference, is that an element $\gamma\in \tM$ with $\chi_{{}^t\gamma^{-1}\gamma}=(T+1)^{2n}(T-1)$ has a semisimple part that is directly $M$-conjugate to $\gamma_t$ for some $t\in F^{\times}$. Then, using a similar homogeneity argument (this time the difference $\delta(M_{\gamma_t})-\delta(H^*)$ is $0$) we can show that $\rJ$ admits a decomposition like \eqref{roughdecomp I} and we conclude in exactly the same way using the fact that $f^{\tM}$ and $\chi(a)L(a)f^{\tM}$ have the same transfers for every $a\in A$.
	\end{proof}
	
	\subsection{Local Langlands for $H$ and twisted character identities}
	
	Let ${}^L H$ be the $L$-group of $H$. More precisely, we take the following finite Galois form of the $L$-group 
	$$\displaystyle {}^L H=\mathrm{SO}_d(\bC)\rtimes \Gal(E/F)$$
	where we recall that the dual group of $H$ is $\widehat{H}=\SO_d(\bC)$ and $H$ splits over the extension $E/F$. In the case where $E\neq F$, that is when $H$ is a special orthogonal group and $\disc(q)\neq 1$, the nontrivial element of $\Gal(E/F)$ acts on $\mathrm{SO}_d(\bC)$ by conjugation by an element of $\mathrm{O}_d(\bC)\setminus \mathrm{SO}_d(\bC)$. In particular, we always have a natural embedding ${}^L H\subset \mathrm{O}_d(\bC)$.
	
	As usual, a {\em $L$-parameter} for $H$ is a continuous group homomorphism
	$$\phi: W'_F:=W_F\times \mathrm{SL}_2(\bC)\to {}^L H,$$
	that commutes with the natural projections to $\Gal(E/F)$, sends $W_F$ to semisimple elements and has a restriction to $\mathrm{SL}_2(\bC)$ that is algebraic. For every $L$-parameter $\phi: W'_F\to {}^L H$, we let $Z_{\mathrm{SO}_d(\bC)}(\phi)$ and $Z_{\mathrm{O}_d(\bC)}(\phi)$ be the centralizers of the image of $\phi$ in $\mathrm{SO}_d(\bC)$ and $\rO_d(\bC)$ respectively and set
	$$\displaystyle S_\phi:= \pi_0(Z_{\mathrm{SO}_d(\bC)}(\phi)),\;\;\; S_\phi^+=\pi_0(Z_{\mathrm{O}_d(\bC)}(\phi))$$
	for their respective component groups.
	
	The usual equivalence relation on $L$-parameters is that of conjugation by the dual group $\widehat{H}=\SO_d(\bC)$ and we denote by $\Phi(H)$ the set of $\SO_d(\bC)$-conjugacy classes of $L$-parameters for $H$. We also denote by $\Phi^+(H)$ the set of $\mathrm{O}_d(\bC)$-conjugacy classes of $L$-parameters for $H$. Note that we have a canonical surjection $\Phi(H)\twoheadrightarrow \Phi^+(H)$ and that $\Phi^+(H)=\Phi(H)$ when $H$ is symplectic (since then $d$ is odd and $\mathrm{O}_d(\bC)=\mathrm{SO}_d(\bC)\times \{ \pm \rI_d\}$). However, this quotient is strict in the orthogonal case. More precisely, the fiber of the map $\Phi(H)\twoheadrightarrow \Phi^+(H)$ above the equivalence class of a $L$-parameter $\phi$ is of cardinality $\frac{2\lvert S_\phi\rvert}{\lvert S_\phi^+\rvert}$, this ratio being equal to $1$ or $2$.

	A $L$-parameter $\phi: W'_F\to {}^L H$ is {\em bounded} if $\phi(W_F)$ is a bounded subgroup. We denote by $\Phi_{\bdd}(H)$ and $\Phi_{\bdd}^+(H)$ the subsets of equivalence classes of bounded $L$-parameters.
	
	The local Langlands correspondence of \cite{artbook}, \cite{MoeLLC} \cite{MoeReLLC} associates to every $\phi\in \Phi_{\bdd}^+(H)$ finite subsets\footnote{We have the equality of $L$-groups ${}^L H^*={}^L H$ so that $L$-parameters for $H$ are the same thing as $L$-parameters for $H^*$.}
	$$\displaystyle \Pi^{H,+}_\phi\subset \Temp(H) \mbox{ and } \Pi^{H^*,+}_\phi\subset \Temp(H^*)$$
	that are $\theta$- and $\theta^*$-invariant respectively and that induce a disjoint decomposition of the tempered duals of $H$ and $H^*$:
	\begin{equation}\label{disjoint Lpackets}
		\displaystyle \Temp(H)=\bigsqcup_{\phi\in \Phi_{\bdd}^+(H)}  \Pi^{H,+}_\phi,\;\;\; \Temp(H^*)=\bigsqcup_{\phi\in \Phi_{\bdd}^+(H)}  \Pi^{H^*,+}_\phi.
	\end{equation}
	For each $\sigma\in \Temp(H)$ (resp. $\sigma^*\in \Temp(H^*)$), we shall denote by $\phi_\sigma\in \Phi^+_{\bdd}(H)$ (resp. $\phi_{\sigma^*}\in \Phi^+_{\bdd}(H)$) its $L$-parameter, that is characterized by $\sigma\in \Pi^{H,+}_{\phi_\sigma}$ (resp. $\sigma^*\in \Pi^{H^*,+}_{\phi_{\sigma^*}}$). Moreover, for every $\sigma\in \Temp(H)$, we set
	$$\displaystyle S_\sigma:=S_{\phi_\sigma}.$$
	
	By the natural embedding ${}^L H\subset \mathrm{GL}_d(\bC)$, a $L$-parameter $\phi\in \Phi_{\bdd}^+(H)$ can also be seen as a semisimple $d$-dimensional representation of $W'_F$ and therefore corresponds, by the local Langlands correspondence for $\GL_d$, to a tempered irreducible representation $\pi_{\phi}^{\GL}$ of $M=\GL_d(F)$. This identifies $\Phi^+_{\bdd}(H)$ with the set of equivalence classes of continuous, semisimple bounded orthogonal representations of the Weil-Deligne group $W'_F$ of determinant $\chi$ (that we identify with a character of $W_F$ via local class field theory). Thus, the map $\phi\mapsto \pi_\phi^{\GL}$ induces a bijection
	$$\displaystyle \Phi^+_{\bdd}(H)\simeq \Temp^{\orth}_\chi(M).$$
	For $\pi\in \Temp_\chi^{\orth}(M)$, we set $S_\pi=S_\phi$ and $S^+_\pi=S_\phi^+$ where $\phi\in \Phi^+_{\bdd}(H)$ is such that $\pi\simeq \pi_\phi^{\GL}$. Note that this coincides with the definition of $S_\pi^+$ given in \S \ref{S def Spi+}.
	
	By the decomposition \eqref{disjoint Lpackets}, we can also view the assignment $\phi\mapsto \pi^{\GL}_\phi$ as ``functorial lifts''
	$$\displaystyle \Temp(H)\to \Temp^{\orth}_\chi(M) \mbox{ and } \Temp(H^*)\to \Temp^{\orth}_\chi(M)$$
	that we shall denote by $\sigma\mapsto \sigma^{\GL}$ and $\sigma^*\mapsto (\sigma^*)^{\GL}$ respectively. More precisely, for every $\sigma\in \Temp(H)$ (resp. every $\sigma^*\in \Temp(H^*)$) we have $\sigma^{\GL}=\pi^{\GL}_{\phi}$ (resp. $(\sigma^*)^{\GL}=\pi^{\GL}_{\phi}$) where $\phi =\phi_\sigma$ (resp. $\phi=\phi_{\sigma^*}$) is the $L$-parameter of $\sigma$ (resp. of $\sigma^*$). By the compatibility of the local Langlands correspondence with parabolic induction, the irreducible summands of the induction of a discrete series of $H$ (resp. $H^*$) belong to the same $L$-packet $\Pi_\phi^H$ (resp. $\Pi^{H^*}_{\phi}$). Hence, the previous maps descend to
	$$\displaystyle \Temp_{\ind}(H)\to \Temp_\chi^{\orth}(M) \mbox{ and } \Temp_{\ind}(H^*)\to \Temp_\chi^{\orth}(M).$$
	
	Let $\Ad_H$ be the adjoint representation of ${}^L H$ on its Lie algebra. Then, for every $\sigma\in \Temp(H)$, we let
	$$\displaystyle \gamma(s,\sigma,\Ad_H,\psi)=\gamma(s,\Ad_H\circ \phi_\sigma,\psi)$$
	be the corresponding Artin $\gamma$-factor. We also set
	$$\displaystyle  \gamma^*(0,\sigma,\Ad_H,\psi)=\left(\zeta_F(s)^{n_\sigma} \gamma(s,\sigma,\Ad_H,\psi) \right)_{s=0}$$
	where $n_\sigma\geq 0$ is the order of the zero of $\gamma(s,\sigma,\Ad_H,\psi)$ at $s=0$. Because $\Ad_H$ is also the restriction of the exterior square of the standard representation by the embedding ${}^L H\subset \GL_d(\bC)$, we have
	\begin{equation}\label{eq gamma factors}
	\displaystyle \gamma(s,\sigma,\Ad_H,\psi)=\gamma(s,\sigma^{\GL},\bigwedge^2,\psi),\;\; \gamma^*(0,\sigma,\Ad_H,\psi)=\gamma^*(0,\sigma^{\GL},\bigwedge^2,\psi).
	\end{equation}
	
	Recall that every representation $\pi\in \Temp^{\orth}(M)$ extends to a representation $\tpi$ of the twisted space $\tM$ and that we have fixed such an extension in Section \ref{Sect ortho reps}, requiring in addition $\tpi$ to be unitary.
	 
 The following character relations characterize completely the $L$-packets $\Pi_\phi^{H,+}$ and $\Pi_{\phi}^{H^*,+}$ in terms of the local Langlands correspondence for $M$ \cite{artbook} \cite{MoeReLLC}.
	
	\begin{theo}[Arthur, M\oe{}glin-Renard]\label{theo twisted character relations}
		There exists a constant $c>0$ such that for every $\pi\in \Temp_\chi^{\orth}(M)$,
		\begin{enumerate}[(i)]
			\item we have
			\begin{equation}
				\displaystyle \sum_{\sigma\in \Temp(H), \; \sigma^{\GL}=\pi} \Theta_\sigma(f^H)=\sum_{\sigma^*\in \Temp(H^*),\; (\sigma^*)^{\GL}=\pi} \Theta_{\sigma^*}(f^{H^*})
			\end{equation}
			for every pair of matching functions $(f^H,f^{H^*})\in C_c^\infty(H)\times C_c^\infty(H^*)$.
			\item  There exists a norm $1$ scalar $u(\pi)\in \mathbb{S}^1$ such that we have
			\begin{equation}
				\displaystyle \Theta_{\widetilde{\pi}}(f^{\tM})=cu(\pi)\frac{\lvert S_\pi^+\rvert}{2\lvert S_\pi\rvert}\sum_{\sigma^*\in \Temp(H^*),\; \sigma^{\GL}=\pi} \Theta_{\sigma^*}(f^{H^*})
			\end{equation}
			for every pair of matching test functions $(f^{\tM},f^{H^*})\in C_c^\infty(\tM)\times C_c^\infty(H^*)$.
		\end{enumerate}

	\end{theo}
	
	\begin{rem}
		The constants $c$ and $u(\pi)$ are here to compensate the freedom in the choice of transfer factors and the extension of $\pi$ to the twisted representation $\widetilde{\pi}$ respectively. Provided these are normalised in a coherent way, using some Whittaker datum, they can actually both be taken to be equal to $1$.
	\end{rem}
	
	\begin{rem}
		Let $\phi\in \Phi^+_{\bdd}(H)$ be the parameter of $\pi$. Then, as we already remarked, the ratio $\frac{\lvert S_\pi^+\rvert}{2\lvert S_\pi\rvert}$ is equal to $1$ when the fiber above $\phi$ in $\Phi_{\bdd}(H)$ is a singleton and $\frac{1}{2}$ otherwise, that is when this fiber is of cardinality $2$.
	\end{rem}
	
	\subsection{Comparison of spectral measures}

	\begin{prop}\label{prop comparison spectral measures}
		The lifting $\Temp_{\ind}(H)\to \Temp_\chi^{\orth}(M)$, $\sigma\mapsto \sigma^{\GL}$, is locally measure-preserving when we equip $\Temp_{\ind}(H)$ with the spectral measure $d\sigma$ and $\Temp_\chi^{\orth}(M)$ with the measure $d^{\orth} \pi$.
	\end{prop}
	
	\begin{proof}
		This just follows from the compatibility of the lifting with parabolic induction and inspection.
	\end{proof}

	\subsection{Proof of the main theorem}
	
	Let us recall the statement.
	
	\begin{theo}
		For almost all $\sigma\in \Temp(H)$, we have the identity
		$$\displaystyle \mu_H(\sigma)=\frac{\lvert \gamma^*(0,\sigma,\Ad_H,\psi)\rvert}{\lvert S_\sigma\rvert}.$$
	\end{theo}
	
	\begin{proof}
		Let $f^H\in C_c^\infty(H)$. Let $f^{H^*}\in C_c^\infty(H^*)$ be a transfer of $f^H$ and let
		$$f_0^{H^*}:=\frac{f^{H^*}+(f^{H^*})^{\theta^*}}{2}$$
		be its projection to the subspace of $\theta^*$-invariant functions. Then, by Theorem \ref{prop image transfer}, $f_0^{H^*}$ is the transfer of a function $f^{\tM}\in C_c^\infty(\tM)$. By Proposition \ref{prop transfer dirac}, and the fact that $\theta^*(\epsilon)=\epsilon$, there exists an absolute constant $C_1\in \mathbb{C}^\times$ such that
		\begin{equation*}
			\displaystyle f^H(\epsilon)=f^{H^*}(\epsilon)=f_0^{H^*}(\epsilon)=C_1 \rI_{\chi}(f^{\tM}).
		\end{equation*}
		By Theorem \ref{theo spectraldec orbint}, this implies that
\begin{equation}\label{eq1 final proof}
			\displaystyle f^H(\epsilon)=C_2\int_{\Temp_\chi^{\orth}(M)} \Theta_{\widetilde{\pi}}(f^{\tM}) u_1(\pi) \frac{\gamma^*(0,\pi,\wedge^2,\psi)}{\lvert S_\pi^{\orth}\rvert} d\pi,
\end{equation}
	where $C_2\in \bC^\times$ is another absolute constant and $u_1(\pi)$ are complex numbers of absolute value $1$. By the character relations of Theorem \ref{theo twisted character relations}, and the fact that the $L$-packets $\Pi^{H^*,+}_\phi$ are $\theta^*$-invariant, there exist an absolute constant $C_3>0$ as well as some constants $u_2(\pi)$ of absolute value $1$ such that
	\begin{equation*}
	\displaystyle \Theta_{\widetilde{\pi}}(f^{\tM})=C_3u_2(\pi)\sum_{\sigma^*\in \Temp(H^*),\; (\sigma^*)^{\GL}=\pi} \Theta_{\sigma^*}(f^{H^*})=C_3u_2(\pi)\sum_{\sigma\in \Temp(H), \; \sigma^{\GL}=\pi} \Theta_\sigma(f^H)
	\end{equation*} 
for every $\pi\in \Temp^{\orth}_\chi(M)$. Combining these identities with \eqref{eq gamma factors} and Proposition \ref{prop comparison spectral measures}, we see that the identity \eqref{eq1 final proof} can be rewritten as
	\begin{equation}\label{eq2 final proof}
	\displaystyle f^H(\epsilon)=C\int_{\Temp_{\ind}(H)} \Theta_{\sigma}(f^{H}) u(\sigma) \frac{\gamma^*(0,\sigma, \Ad_H,\psi)}{\lvert S_\sigma\rvert} d\sigma.
	\end{equation}
for some absolute constant $C>0$ and complex numbers $u(\sigma)$ of absolute values $1$. (Note that, by definition, $S_\sigma=S^{\orth}_{\sigma^{\GL}}$.) On the other hand, by the Plancherel formula for $H$, we have
\begin{equation}\label{eq3 final proof}
\displaystyle f^H(\epsilon)=\int_{\Temp_{\ind}(H)} \Theta_{\sigma}(f^{H}) \omega_\sigma(\epsilon) \mu_H(\sigma) d\sigma.
\end{equation}
Comparing \eqref{eq2 final proof} with \eqref{eq3 final proof}, and by the unicity of the Plancherel measure, we deduce that for almost all $\sigma\in \Temp_{\ind}(H)$,
\begin{equation*}
\displaystyle \mu_H(\sigma)=C \omega_\sigma(\epsilon)u(\sigma) \frac{\gamma^*(0,\sigma, \Ad_H,\psi)}{\lvert S_\sigma\rvert}.
\end{equation*}
By positivity of the Plancherel density, this gives
$$\mu_H(\sigma)=C\frac{\lvert \gamma^*(0,\sigma, \Ad_H,\psi)\rvert}{\lvert S_\sigma\rvert}.$$

Finally, for $\sigma=\St\in \Pi_2(H)$ the Steinberg representation, by \cite[\S 3.3]{HII}, we know that
$$\displaystyle \mu_H(\St)=d_H(\St)=\frac{\lvert \gamma(0,\St, \Ad_H,\psi)\rvert}{\lvert S_\sigma\rvert}.$$
Therefore, $C=1$ and the theorem is proved.
	\end{proof}

	\appendix
	
	\section{A spectral limit}
	
	The goal of this appendix is to give the proof of Proposition \ref{prop1 spectral limit}. We adopt the notations of Section \ref{Sect ortho reps}. Namely, $V$ is a vector space over $F$ of finite dimension $d$, $M=\GL(V)$, $A$ is the center of $M$, we fix a quadratic character $\chi: A\to \{\pm 1\}$, $\Temp_\chi(M)$ is the set of isomorphism classes of irreducible tempered representations $\pi$ with central character $\omega_\pi=\chi$ and $\Temp_\chi^{\orth}(M)$ denotes the subset of those $\pi\in \Temp_\chi(M)$ which are orthogonal, i.e. whose Langlands parameter factors through an orthogonal group. We equip $\Temp_\chi(M)$ and $\Temp_\chi^{\orth}(M)$ with elementary measures, both denoted by $d\pi$, as in \S \ref{S spectral measure} and \S \ref{S orth spectral measure}. We also write $d\mu_{M,\chi}(\pi)=\mu_{M,\chi}(\pi)d\pi$ for the Plancherel measure on $\Temp_{\chi}(M)$ normalised as in \S \ref{S Plancherel}. For convenience, we recall the statement of Proposition \ref{prop1 spectral limit} below.
	
	\begin{prop}\label{prop spectral limit}
		Let $\Phi\in C_c^\infty(\Temp_\chi(M))$. Then, we have
		\[\begin{aligned}
			\displaystyle & \lim\limits_{s\to 0^+} d\gamma(s,\mathbf{1}_F,\psi)\int_{\Temp_\chi(M)} \Phi(\pi) \gamma(s,\pi,\Sym^2,\psi)^{-1} \mu_{M,\chi}(\pi)d_\chi\pi= \\
			& 2\chi(-1)^{d-1}\int_{\Temp^{\orth}_\chi(M)} \Phi(\pi) \frac{\gamma^*(0,\pi,\wedge^2,\psi)}{\lvert S_\pi^{\orth}\rvert} d^{\orth}\pi.
		\end{aligned}\]
	\end{prop}
	
	\begin{proof}
		This can be proved similarly to \cite[Proposition 3.41]{BPPlanch} essentially by brute force reducing everything to the computation  of the residues of some explicit families of distributions on $\bR^n$ which has been done in \cite[Proposition 3.31]{BPPlanch}. For the reader convenience, we present the details of this reduction below, mainly because keeping track of the various normalizations of measures can be tedious and also in order to convince oneself that the order of the component group appearing on the right hand side is the correct one.
		
		By the formula \eqref{formula muMchi} for the Plancherel density of $M$, for $\Re(s)>0$ we have
		\[\begin{aligned}
			\displaystyle & \int_{\Temp_\chi(M)} \Phi(\pi) \gamma(s,\pi,\Sym^2,\psi)^{-1} \mu_{M,\chi}(\pi)d_\chi\pi \\
			& =d^{-1}\chi(-1)^{d-1}\int_{\Temp_\chi(M)} \Phi(\pi) \gamma(s,\pi,\Sym^2,\psi)^{-1} \gamma^*(0,\pi, \Ad_{M/A},\psi)d_\chi\pi.
		\end{aligned}\]
		Fix $\pi\in \Temp_\chi(M)$. Using a partition of unity, we may assume that $\Phi$ is supported in a small neighborhood of $\pi$. More precisely, writing $\pi$ as a parabolic induction $\rI_L^M(\sigma)$ where $L$ is a Levi subgroup of $M$ and $\sigma\in \Pi_2(L)$, we may assume that there is a $W(M,\sigma)$-invariant open neighborhood $\cU\subset i(\cA_L^M)^*$ of $0$ such that $\lambda\mapsto \pi_\lambda:=\rI_L^M(\sigma\otimes \lambda)$ identifies $\cU/W(M,\sigma)$ with an open neighborhood of $\pi$ in which $\Phi$ is supported. Then, by definition of the measure $d_\chi\pi$, we have
		\begin{align}\label{eq8}
			\displaystyle & \int_{\Temp_\chi(M)} \Phi(\pi) \gamma(s,\pi,\Sym^2,\psi)^{-1} \mu_{M,\chi}(\pi)d_\chi\pi= \\
			\nonumber & d^{-1}\frac{\chi(-1)^{d-1}}{\lvert W(M,\sigma)\rvert}\int_{i(\cA_L^M)^*} \phi(\lambda) \gamma(s,\pi_\lambda,\Sym^2,\psi)^{-1} \gamma^*(0,\pi_\lambda, \Ad_{M/A},\psi)d\lambda
		\end{align}
		where $\phi\in C_c^\infty(i(\cA_L^M)^*)$ is the function that equals $\lambda\mapsto \Phi(\pi_\lambda)$ on $\cU$ and vanishes outside $\cU$.
		
		Further, we may decompose $L$ as a product
		$$\displaystyle L= \prod_{i\in I^n} \GL_{d_i}(F)^{m_i+n_i}\times \prod_{j\in I^s} \GL_{e_j}(F)^{p_j}\times \prod_{k\in I^o} \GL_{f_k}(F)^{q_k},$$
		and accordingly $\sigma$ as an exterior tensor product
		$$\displaystyle \sigma= \bigboxtimes_{i\in I^n} \sigma_i^{\boxtimes m_i}\boxtimes (\sigma_i^\vee)^{\boxtimes n_i} \boxtimes \bigboxtimes_{j\in I^s} \sigma_j^{\boxtimes p_j}\boxtimes \bigboxtimes_{k\in I^o} \sigma_k^{\boxtimes q_k},$$
		where $I^n$, $I^s$ and $I^o$ are disjoint finite sets and $\sigma_i\in \Pi_2(\GL_{d_i}(F))$, $\sigma_j\in \Pi_2(\GL_{e_j}(F))$ and $\sigma_k\in \Pi_2(\GL_{f_k}(F))$ are such that:
		\begin{itemize}
			\item For $i\in I^n$, $\sigma_i$ is not selfdual;
			
			\item For $j\in I^s$, $\sigma_j$ is selfdual of symplectic type;
			
			\item For $k\in I^o$, $\sigma_k$ is selfdual of orthogonal type.
		\end{itemize}
		Note that $W(M,\sigma)$ can naturally be identified with the product of symmetric groups
		$$\displaystyle W:=\prod_{i\in I^n} \mathfrak{S}_{m_i}\times \mathfrak{S}_{n_i}\times \prod_{j\in I^s} \mathfrak{S}_{p_j}\times \prod_{k\in I^o} \mathfrak{S}_{q_k}.$$
		Set
		\begin{equation}\label{eq9}
			\cA^*=\prod_{i\in I^n} \bR^{m_i+n_i}\times \prod_{j\in I^s} \bR^{p_j}\times \prod_{k\in I^o} \bR^{q_k}
		\end{equation}
		and let $\cA_0^*\subset \cA^*$ be the subspace of vectors whose sum of coordinates vanish. We will use the identification $\cA^*\simeq \cA_L^*$ sending the standard basis of $\cA^*$ onto the standard basis of $X^*(A_L)$. This induces an isomorphism $\cA_0^*\simeq (\cA_L^M)^*$. We equip $\cA^*$ with the product of Lebesgue measures and $\cA_0^*$ with the unique Haar measure inducing on the quotient $\cA^*/\cA_0^*\simeq \bR$ the Lebesgue measure (where the identification of the quotient with $\bR$ is given by summing the coordinates). Thus, by our choice of Haar measure on $(\cA_L^M)^*$ (see \S \ref{S measures}), the previous identification $\cA_0^*\simeq (\cA_L^M)^*$ has Jacobian $dP^{-1}\left(\frac{2\pi}{\log(q)}\right)^{1-S}$ where
		$$\displaystyle P:=[X^*(A_L):X^*(L)]=\prod_{i\in I^n} d_i^{m_i+n_i}\times \prod_{j\in I^s} e_j^{p_j}\times \prod_{k\in I^o} f_k^{q_k}$$
		and
		$$S=\sum_{i\in I^n} m_i+n_i+\sum_{j\in I^s} p_j +\sum_{k\in I^o} q_k.$$
		In particular, we can rewrite the right hand side of \eqref{eq8} as
		\begin{equation}\label{eq12}
			\displaystyle \left(\frac{2\pi}{\log(q)}\right)^{1-S}\frac{\chi(-1)^{d-1}}{P \lvert W\rvert}\int_{i\cA_0^*} \phi(\lambda) \gamma(s,\pi_\lambda,\Sym^2,\psi)^{-1} \gamma^*(0,\pi_\lambda, \Ad_{M/A},\psi)d\lambda
		\end{equation}
		For $\lambda\in i\cA^*$, let us denote its coordinates in the product decomposition \eqref{eq9} by
		$$\displaystyle (x_{i,\ell}(\lambda))_{\substack{i\in I^n \\ 1\leq \ell\leq m_i}} \;\;\; (x^\vee_{i,\ell}(\lambda))_{\substack{i\in I^n \\ 1\leq \ell\leq n_i}}\;\;\; (y_{j,\ell}(\lambda))_{\substack{j\in I^s \\ 1\leq \ell\leq p_j}}\;\;\; (z_{k,\ell}(\lambda))_{\substack{k\in I^o \\ 1\leq \ell\leq q_k}}.$$
		Then, for $\lambda\in i\cA^*$, we have
		$$\displaystyle \sigma_{\lambda}= \bigboxtimes_{i\in I^n} \bigboxtimes_{\ell=1}^{m_i}\sigma_i\lvert \det\rvert^{x_{i,\ell}(\lambda)/d_i}\boxtimes \bigboxtimes_{\ell=1}^{n_i} \sigma_i^\vee\lvert \det\rvert^{x^\vee_{i,\ell}(\lambda)/d_i} \boxtimes \bigboxtimes_{j\in I^s} \bigboxtimes_{\ell=1}^{p_j}\sigma_j \lvert \det\rvert^{y_{j,\ell}(\lambda)/e_j}\boxtimes \bigboxtimes_{k\in I^o} \bigboxtimes_{\ell=1}^{q_k} \sigma_k\lvert \det\rvert^{z_{k,\ell}(\lambda)/f_k}.$$
		Let us identify $\bC$ with the line in $\cA_{\bC}^*=\cA^*\otimes_{\bR} \bC$ generated by $\det\in X^*(M)\subset \cA_L^*\simeq \cA^*$. Using inductivity of symmetric square $\gamma$-factors as well as the localization of their zeroes, we find that
		\begin{align}\label{eq10}
			\displaystyle & \gamma(s,\pi_\lambda,\Sym^2,\psi)^{-1}= \\
			\nonumber & \prod_{i\in I^n} \prod_{\substack{1\leq \ell\leq m_i \\ 1\leq \ell'\leq n_i}} (s+\frac{x_{i,\ell}(\lambda)+x_{i,\ell'}^\vee(\lambda)}{d_i})^{-1}\times \prod_{j\in I^s} \prod_{1\leq \ell<\ell'\leq p_j} (s+\frac{y_{j,\ell}(\lambda)+y_{j,\ell'}(\lambda)}{e_j})^{-1} \\
			\nonumber & \times \prod_{k\in I^o} \prod_{1\leq \ell\leq \ell'\leq q_k} (s+\frac{z_{k,\ell}(\lambda)+z_{k,\ell'}(\lambda)}{f_k})^{-1} \times G(\frac{s}{2}+\lambda)
		\end{align}
		where $G$ is a meromorphic function on $\cA_{\bC}^*$ which is regular at $0$. Similarly, we have
		\begin{align}\label{eq11}
			\displaystyle & \gamma^*(0,\pi_\lambda, \Ad_{M/A},\psi)= \\
			\nonumber & \prod_{i\in I^n} \prod_{1\leq \ell\neq \ell'\leq m_i} (\frac{x_{i,\ell}(\lambda)-x_{i,\ell'}(\lambda)}{d_i})\prod_{1\leq \ell\neq \ell'\leq n_i} (\frac{x^\vee_{i,\ell}(\lambda)-x^\vee_{i,\ell'}(\lambda)}{d_i}) \\
			\nonumber & \prod_{j\in I^s} \prod_{1\leq \ell\neq \ell'\leq p_j} (\frac{y_{j,\ell}(\lambda)-y_{j,\ell'}(\lambda)}{e_j}) \times \prod_{k\in I^o} \prod_{1\leq \ell\neq \ell'\leq q_k} (\frac{z_{k,\ell}(\lambda)-z_{k,\ell'}(\lambda)}{f_k}) \times H(\lambda)
		\end{align}
		where $H$ is a meromorphic function on $\cA_{\bC}^*$ which is regular at $0$.
		
		From \eqref{eq10} and \eqref{eq11}, we deduce that, provided $\cU$ is small enough, the function
		$$\displaystyle \lambda\in i\cA_0^*\mapsto \phi(\lambda)\gamma(s,\pi_\lambda,\Sym^2,\psi)^{-1}\gamma^*(0,\pi_\lambda,\Ad_{M/A},\psi)$$
		can be written as a product
		$$\displaystyle \prod_{i\in I^n} P_{m_i,n_i,s}(\frac{\underline{x}_i(\lambda)}{d_i})\prod_{j\in I^s} Q_{p_j,s}(\frac{\underline{y}_j(\lambda)}{e_j})\prod_{k\in I^o} R_{q_k,s}(\frac{\underline{z}_k(\lambda)}{f_k})\phi_s(\lambda)$$
		where $s\mapsto \phi_s\in C_c^\infty(i\cA_0^*)$ is a map that is continuous at $s=0$ and the rational functions $P_{m_i,n_i,s}$, $Q_{p_j,s}$, $R_{q_k,s}$ are as in \cite[\S 3.2]{BPPlanch}. Then, according to \cite[Proposition 3.31]{BPPlanch}, the product of \eqref{eq12} with $d\gamma(s,\mathbf{1}_F,\psi)$ has a limit as $s\to 0$ which vanishes unless $m_i=n_i$ for every $i\in I^n$ and $p_j$ is even for every $j\in I^s$, i.e. unless $\pi\in \Temp_\chi^{\orth}(M)$. Moreover, when $\pi\in \Temp_\chi^{\orth}(M)$ this limit admits the following expression\footnote{Note that the parameter $n$ in loc.\ cit.\ corresponds to $d$ in our computation.}
		\begin{align}\label{eq16}
			\displaystyle & \log(q)\gamma^*(0,\mathbf{1}_F,\psi) \left(\frac{2\pi}{\log(q)}\right)^{1-S}\frac{\chi(-1)^{d-1}}{P \lvert W'\rvert} D (2\pi)^{N-1} 2^{1-c} \times \\
			\nonumber & \int_{i\cA'} \phi(\mu) \lim\limits_{s\to 0} s^N\gamma(s,\pi_\mu,\Sym^2,\psi)^{-1} \gamma^*(0,\pi_\mu, \Ad_{M/A},\psi)d\mu
		\end{align}
		where:
		\begin{itemize}
			\item $W'$ is the following subgroup of $W$:
			$$\displaystyle W'=\prod_{i\in I^n} \mathfrak{S}_{n_i}\times \prod_{j\in I^s} \mathfrak{S}_{p_j/2}\ltimes (\bZ/2)^{p_j/2}\times \prod_{k\in I^o} \mathfrak{S}_{\floor{\frac{q_k}{2}}}\ltimes (\bZ/2)^{\floor{\frac{q_k}{2}}};$$
			
			\item $D=\prod_{i\in I^n} d_i^{n_i}\times \prod_{j\in I^s} e_j^{\frac{p_j}{2}}\times \prod_{k\in I^o} f_k^{\ceil{\frac{q_k}{2}}}$;
			
			\item $N=\sum_{i\in I^n} n_i+\sum_{j\in I^s} \frac{p_j}{2}+\sum_{k\in I^o}\ceil{\frac{q_k}{2}}$;
			
			\item $c=\sharp \left\{k\in I^o\mid q_k \mbox{ is odd } \right\}$
			
			\item $\cA'$ is the subspace of those $\lambda\in \cA$ such that $x_{i,\ell}(\lambda)+x_{i,\ell}^\vee(\lambda)=0$ for every $i\in I^n$, $1\leq \ell\leq n_i$, $y_{j,\ell}(\lambda)+y_{j,p_j+1-\ell}(\lambda)=0$ for every $j\in I^s$, $1\leq \ell\leq p_j$, and $z_{k,\ell}(\lambda)+z_{k,q_k+1-\ell}(\lambda)=0$ for every $k\in I^o$, $1\leq \ell\leq q_k$. Moreover, we equip $\cA'$ with the Haar measure corresponding to the product of Lebesgue measures by the isomorphism
			$$\displaystyle \cA'\simeq \prod_{i\in I^n} \bR^{m_i}\times \prod_{i\in I^s} \bR^{p_j/2}\times \prod_{k\in I^o} \bR^{\floor{\frac{q_k}{2}}}$$
			$$\displaystyle \lambda\mapsto \left((x_{i,\ell}(\lambda))_{\substack{i\in I^n \\ 1\leq \ell\leq m_i}},(y_{j,\ell}(\lambda))_{\substack{j\in I^s \\ 1\leq \ell\leq p_j/2}} , (z_{k,\ell}(\lambda))_{\substack{k\in I^o \\ 1\leq \ell\leq \floor{\frac{q_k}{2}}}}\right).$$
		\end{itemize}
		Note that, e.g. by \eqref{eq10}, for $\mu\in i\cA'$ in general position the $\gamma$-factor $\gamma(s,\pi_\mu,\Sym^2,\psi)$ has a zero of order $N$ at $s=0$ so that
		\begin{equation*}
			\displaystyle \lim\limits_{s\to 0} s^N \gamma(s,\pi_\mu,\Sym^2,\psi)^{-1}=\log(q)^{-N}\gamma^*(0,\pi_\mu,\Sym^2,\psi)^{-1}.
		\end{equation*}
		Recalling that
		$$\gamma^*(0,\mathbf{1}_F,\psi)\gamma^*(0,\pi_\mu, \Ad_{M/A},\psi)=\gamma^*(0,\pi_\mu\times \pi_\mu^\vee, \psi)=\gamma^*(0,\pi_\mu,\Sym^2,\psi)\gamma^*(0,\pi_\mu,\wedge^2,\psi)$$
		this leads to
		\begin{equation}\label{eq13}
			\displaystyle \gamma^*(0,\mathbf{1}_F,\psi)\lim\limits_{s\to 0} s^N \gamma(s,\pi_\mu,\Sym^2,\psi)^{-1}\gamma^*(0,\pi_\mu, \Ad_{M/A},\psi)=\log(q)^{-N} \gamma^*(0,\pi_\mu,\wedge^2,\psi).
		\end{equation}
		Moreover, we may find an involution $w\in W(M,L)$ such that $\sigma\in \Pi_2^{\orth,w}(L)$ and the isomorphism $\cA\simeq \cA_L^*$ identifies $\cA'$ with $\cA_{L,w}^{*}$. As before, we readily check that this identification has Jacobian equal to $\frac{D}{P}\left(\frac{2\pi}{\log(q)}\right)^{N-S}$ and that
		\begin{equation}\label{eq14}
			\displaystyle W'\simeq W(M,\sigma)_w,\;\; \lvert S_\pi^{\orth}\rvert=2^c.
		\end{equation}
		Thus, by \eqref{eq13} and \eqref{eq14}, we see that \eqref{eq16} is equal to
		\[\begin{aligned}
			\displaystyle & \frac{2\chi(-1)^{d-1}}{\lvert W(M,\sigma)_w\rvert \lvert S_\pi^{\orth}\rvert} \frac{D}{P}\left(\frac{2\pi}{\log(q)} \right)^{N-S} \int_{i\cA'} \phi(\mu) \gamma^*(0,\pi_\mu,\wedge^2,\psi) d\mu \\
			& =\frac{2\chi(-1)^{d-1}}{\lvert W(M,\sigma)_w\rvert \lvert S_\pi^{\orth}\rvert} \int_{i\cA_{L}^{*,w}} \phi(\mu) \gamma^*(0,\pi_\mu,\wedge^2,\psi) d\mu.
		\end{aligned}\]
		However, provided $\cU$ was chosen small enough, this last expression is the same as the right hand side of the proposition.
	\end{proof}
	
	\bibliography{biblio}

\begin{thebibliography}{GGP12}

\bibitem[Art13]{artbook}
James Arthur.
\newblock {\em The endoscopic classification of representations}, volume~61 of
  {\em American Mathematical Society Colloquium Publications}.
\newblock American Mathematical Society, Providence, RI, 2013.
\newblock Orthogonal and symplectic groups.

\bibitem[BP21]{BPPlanch}
R.~Beuzart-Plessis.
\newblock Plancherel formula for {${\rm GL}_n(F)\backslash {\rm GL}_n(E)$} and
  applications to the {I}chino-{I}keda and formal degree conjectures for
  unitary groups.
\newblock {\em Invent. Math.}, 225(1):159--297, 2021.

\bibitem[CHH88]{CoHaHo}
M.~Cowling, U.~Haagerup, and R.~Howe.
\newblock Almost {$L^2$} matrix coefficients.
\newblock {\em J. Reine Angew. Math.}, 387:97--110, 1988.

\bibitem[Coh]{CohenJ}
Jo\"el Cohen.
\newblock A spectral expression for a certain orbital integral.
\newblock arXiv:1407.4316.

\bibitem[GGP12]{GGP}
W.~T. Gan, B.~Gross, and D.~Prasad.
\newblock Symplectic local root numbers, central critical {$L$} values, and
  restriction problems in the representation theory of classical groups.
\newblock {\em Ast\'{e}risque}, (346):1--109, 2012.

\bibitem[Gol94]{GoldSO}
David Goldberg.
\newblock Reducibility of induced representations for {${\rm Sp}(2n)$} and
  {${\rm SO}(n)$}.
\newblock {\em Amer. J. Math.}, 116(5):1101--1151, 1994.

\bibitem[GR10]{GrossReeder}
Benedict~H. Gross and Mark Reeder.
\newblock Arithmetic invariants of discrete {L}anglands parameters.
\newblock {\em Duke Math. J.}, 154(3):431--508, 2010.

\bibitem[HC99]{HCadminv}
Harish-Chandra.
\newblock {\em Admissible invariant distributions on reductive {$p$}-adic
  groups}, volume~16 of {\em University Lecture Series}.
\newblock American Mathematical Society, Providence, RI, 1999.
\newblock With a preface and notes by Stephen DeBacker and Paul J. Sally, Jr.

\bibitem[Hen00]{HennLLC}
Guy Henniart.
\newblock Une preuve simple des conjectures de {L}anglands pour {${\rm GL}(n)$}
  sur un corps {$p$}-adique.
\newblock {\em Invent. Math.}, 139(2):439--455, 2000.

\bibitem[Hen10]{HenLfn}
Guy Henniart.
\newblock Correspondance de {L}anglands et fonctions {$L$} des carr\'es
  ext\'erieur et sym\'etrique.
\newblock {\em Int. Math. Res. Not. IMRN}, (4):633--673, 2010.

\bibitem[HII08]{HII}
Kaoru Hiraga, Atsushi Ichino, and Tamotsu Ikeda.
\newblock Formal degrees and adjoint {$\gamma$}-factors.
\newblock {\em J. Amer. Math. Soc.}, 21(1):283--304, 2008.

\bibitem[HT01]{HTLLC}
Michael Harris and Richard Taylor.
\newblock {\em The geometry and cohomology of some simple {S}himura varieties},
  volume 151 of {\em Annals of Mathematics Studies}.
\newblock Princeton University Press, Princeton, NJ, 2001.
\newblock With an appendix by Vladimir G. Berkovich.

\bibitem[ILM17]{ILMfd}
Atsushi Ichino, Erez Lapid, and Zhengyu Mao.
\newblock On the formal degrees of square-integrable representations of odd
  special orthogonal and metaplectic groups.
\newblock {\em Duke Math. J.}, 166(7):1301--1348, 2017.

\bibitem[Kot88]{KottTam}
Robert~E. Kottwitz.
\newblock Tamagawa numbers.
\newblock {\em Ann. of Math. (2)}, 127(3):629--646, 1988.

\bibitem[KS99]{KottShel}
Robert~E. Kottwitz and Diana Shelstad.
\newblock Foundations of twisted endoscopy.
\newblock {\em Ast\'erisque}, (255):vi+190, 1999.

\bibitem[Lan76]{LangEis}
Robert~P. Langlands.
\newblock {\em On the functional equations satisfied by {E}isenstein series},
  volume Vol. 544 of {\em Lecture Notes in Mathematics}.
\newblock Springer-Verlag, Berlin-New York, 1976.

\bibitem[LW13]{LabWal}
J.-P. Labesse and J.-L. Waldspurger.
\newblock {\em La formule des traces tordue d'apr\`es le {F}riday {M}orning
  {S}eminar}, volume~31 of {\em CRM Monograph Series}.
\newblock American Mathematical Society, Providence, RI, 2013.
\newblock With a foreword by Robert Langlands [dual English/French text].

\bibitem[Mg14a]{Moegpaquetstable}
Colette M\oe~glin.
\newblock Paquets stables des s\'eries discr\`etes accessibles par endoscopie
  tordue; leur param\`etre de {L}anglands.
\newblock In {\em Automorphic forms and related geometry: assessing the legacy
  of {I}. {I}. {P}iatetski-{S}hapiro}, volume 614 of {\em Contemp. Math.},
  pages 295--336. Amer. Math. Soc., Providence, RI, 2014.

\bibitem[Mg14b]{MoeLLC}
Colette M\oe~glin.
\newblock Paquets stables des s\'eries discr\`etes accessibles par endoscopie
  tordue; leur param\`etre de {L}anglands.
\newblock In {\em Automorphic forms and related geometry: assessing the legacy
  of {I}. {I}. {P}iatetski-{S}hapiro}, volume 614 of {\em Contemp. Math.},
  pages 295--336. Amer. Math. Soc., Providence, RI, 2014.

\bibitem[MR18]{MoeReLLC}
Colette Moeglin and David Renard.
\newblock Sur les paquets d'{A}rthur des groupes classiques et unitaires non
  quasi-d\'eploy\'es.
\newblock In {\em Relative aspects in representation theory, {L}anglands
  functoriality and automorphic forms}, volume 2221 of {\em Lecture Notes in
  Math.}, pages 341--361. Springer, Cham, 2018.

\bibitem[MW16]{MWStab1}
Colette Moeglin and Jean-Loup Waldspurger.
\newblock {\em Stabilisation de la formule des traces tordue. {V}ol. 1}, volume
  316 of {\em Progress in Mathematics}.
\newblock Birkh\"auser/Springer, Cham, 2016.

\bibitem[Ngo10]{Ngo}
Bao~Chau Ngo.
\newblock Le lemme fondamental pour les alg\`ebres de {L}ie.
\newblock {\em Publ. Math. Inst. Hautes \'Etudes Sci.}, (111):1--169, 2010.

\bibitem[RR72]{Rao}
R.~Ranga~Rao.
\newblock Orbital integrals in reductive groups.
\newblock {\em Ann. of Math. (2)}, 96:505--510, 1972.

\bibitem[Sha84]{ShaGLn}
Freydoon Shahidi.
\newblock Fourier transforms of intertwining operators and {P}lancherel
  measures for {${\rm GL}(n)$}.
\newblock {\em Amer. J. Math.}, 106(1):67--111, 1984.

\bibitem[Sha90a]{Sha90}
F.~Shahidi.
\newblock A proof of {L}anglands' conjecture on {P}lancherel measures;
  complementary series for {$p$}-adic groups.
\newblock {\em Ann. of Math. (2)}, 132(2):273--330, 1990.

\bibitem[Sha90b]{ShaLangconj}
Freydoon Shahidi.
\newblock A proof of {L}anglands' conjecture on {P}lancherel measures;
  complementary series for {$p$}-adic groups.
\newblock {\em Ann. of Math. (2)}, 132(2):273--330, 1990.

\bibitem[Sha92]{Shaend}
Freydoon Shahidi.
\newblock Twisted endoscopy and reducibility of induced representations for
  {$p$}-adic groups.
\newblock {\em Duke Math. J.}, 66(1):1--41, 1992.

\bibitem[Sil79]{Silb}
Allan~J. Silberger.
\newblock {\em Introduction to harmonic analysis on reductive {$p$}-adic
  groups}, volume~23 of {\em Mathematical Notes}.
\newblock Princeton University Press, Princeton, NJ; University of Tokyo Press,
  Tokyo, 1979.
\newblock Based on lectures by Harish-Chandra at the Institute for Advanced
  Study, 1971--1973.

\bibitem[Wal97]{WaldLFimptrans}
J.-L. Waldspurger.
\newblock Le lemme fondamental implique le transfert.
\newblock {\em Compositio Math.}, 105(2):153--236, 1997.

\bibitem[Wal03]{WaldPlanch}
J.-L. Waldspurger.
\newblock La formule de {P}lancherel pour les groupes {$p$}-adiques (d'apr\`es
  {H}arish-{C}handra).
\newblock {\em J. Inst. Math. Jussieu}, 2(2):235--333, 2003.

\bibitem[Wal06]{Walchgtcar}
J.-L. Waldspurger.
\newblock Endoscopie et changement de caract\'eristique.
\newblock {\em J. Inst. Math. Jussieu}, 5(3):423--525, 2006.

\bibitem[Wal08]{Waldendtordue}
J.-L. Waldspurger.
\newblock L'endoscopie tordue n'est pas si tordue.
\newblock {\em Mem. Amer. Math. Soc.}, 194(908):x+261, 2008.

\bibitem[Wal10]{Waldtransfact}
J.-L. Waldspurger.
\newblock Les facteurs de transfert pour les groupes classiques: un formulaire.
\newblock {\em Manuscripta Math.}, 133(1-2):41--82, 2010.

\bibitem[Wei82]{Weiladeles}
Andr\'e Weil.
\newblock {\em Adeles and algebraic groups}, volume~23 of {\em Progress in
  Mathematics}.
\newblock Birkh\"auser, Boston, MA, 1982.
\newblock With appendices by M. Demazure and Takashi Ono.

\end{thebibliography}
	\bibliographystyle{alpha}

\end{document}